\documentclass{article}

\usepackage[longnamesfirst,numbers,sort&compress]{natbib}
\makeatletter
\def\NAT@spacechar{~}
\makeatother
\usepackage[bmargin=25mm,tmargin=25mm,lmargin=30mm,rmargin=30mm]{geometry}

\usepackage{amsmath,amssymb,amsthm, mathrsfs, mathtools,bbm,bm}
\usepackage{enumitem}
\usepackage{colortbl}
\usepackage[hypertexnames=false]{hyperref}
\hypersetup{
    colorlinks,
    linkcolor={red!60!black},
    citecolor={green!60!black},
    urlcolor={blue!40!black}
}

\usepackage{subfig}
\usepackage{algorithm} 
\usepackage{algpseudocode}

\setlist[enumerate,1]{label={(\roman*)},noitemsep, topsep=0pt}
\setlist[enumerate,2]{label={(\alph*)},noitemsep, topsep=0pt}
\setlist[enumerate,3]{label={(\arabic*)},noitemsep, topsep=0pt}
\setlist[itemize]{nolistsep,noitemsep, topsep=0pt}

\numberwithin{equation}{section}
\numberwithin{figure}{section}

\usepackage[noabbrev,capitalise]{cleveref}
\crefname{thm}{Theorem}{Theorems}
\crefname{enumi}{}{}
\crefname{equation}{}{}
\crefname{subsection}{Subsection}{Subsections}

\theoremstyle{plain}
\newtheorem{thm}{Theorem}[section]
\newtheorem{theorem}[thm]{Theorem}
\newtheorem*{theorem*}{Theorem}
\newtheorem{proposition}[thm]{Proposition}
\newtheorem{observation}[thm]{Observation}
\newtheorem{claim}[thm]{Claim}
\newtheorem{corollary}[thm]{Corollary}
\newtheorem{conjecture}[thm]{Conjecture}

\theoremstyle{definition}
\newtheorem{definition}[thm]{Definition}
\newtheorem{problem}[thm]{Problem}
\newtheorem*{problem*}{Problem}
\newtheorem{question}[thm]{Question}

\newcommand{\eps}{\varepsilon}

\DeclarePairedDelimiter{\parens}{(}{)}
\DeclarePairedDelimiter{\set}{\{}{\}}
\DeclarePairedDelimiter{\brackets}{[}{]}
\DeclarePairedDelimiter{\floor}{\lfloor}{\rfloor}
\DeclarePairedDelimiter{\ceil}{\lceil}{\rceil}
\DeclarePairedDelimiter\abs{\lvert}{\rvert}  
\DeclarePairedDelimiter\size{\lvert}{\rvert} 

\renewcommand{\Pr}{\mathbb{P}}

\newcommand{\RBP}[1][{}]{\ensuremath{R_{#1}\cup B_{#1}\cup P_{#1}}}
\newcommand{\RBPprime}[1][{}]{\ensuremath{R'_{#1}\cup B'_{#1}\cup P'_{#1}}}
\newcommand{\se}{\ensuremath{\subseteq}}
\newcommand{\sm}{\ensuremath{\setminus}}
\newcommand{\RT}[3]{\ensuremath{\mathrm{{\bf RT}}_{#1}(#3,#2)}}

\allowdisplaybreaks

\usepackage{tikz}
\usepackage{calc}
\usetikzlibrary{calc,angles,quotes,decorations.pathreplacing,calligraphy,backgrounds,arrows,patterns}
\usetikzlibrary{shapes.symbols,shapes.geometric,decorations.markings}

\tikzstyle{vertex}=[circle,draw=black, fill=black, inner sep=0pt, minimum size=6pt]
\tikzstyle{smallvertex}=[circle, draw=black, fill=black, inner sep=0pt, minimum size=4pt]
\tikzstyle{tinyvertex}=[circle, draw=black, fill=gray, inner sep=0pt, minimum size=2pt] 

\newcommand{\vertex}{\node[vertex]}

\newcommand{\tinyvertex}{\node[tinyvertex]} 

\newcounter{Angle}
\newcounter{Neighbour}

\newcommand{\widthedgethick}{1.5}

\definecolor{purple}{rgb}{0.87, 0.0, 1.0}

\newcommand{\ExamplePurpleBlowUp}[1]{
\begin{tikzpicture}[x=#1 cm, y=#1 cm]

\begin{scope}[shift={(3,0)}] 
\draw (0,0) node{$R$};
\foreach \i in {1, 2, ..., 4} {
    \setcounter{Angle}{180-\i * 360 / 4 + 45}
    \vertex(a\i) at (\theAngle:1) [label=:]{};
    \vertex(b\i) at (\theAngle:1.3) [label=:]{};
}
\vertex(c1) at (135:1.6) [label=:]{};
\vertex(c2) at (45:1.6) [label=:]{};

\path
(a1) edge[color=red,line width = \widthedgethick] (a2)
(a2) edge[color=red,line width = \widthedgethick] (a3)
(a3) edge[color=red,line width = \widthedgethick] (a4)
(a4) edge[color=red,line width = \widthedgethick] (a1);

\path
(b1) edge[color=red,line width = \widthedgethick] (b2)
(b2) edge[color=red,line width = \widthedgethick] (b3)
(b3) edge[color=red,line width = \widthedgethick] (b4)
(b4) edge[color=red,line width = \widthedgethick] (b1);

\path
(c1) edge[color=red,line width = \widthedgethick] (c2);
\end{scope}

\begin{scope}[shift={(6,0)}] 
\draw (0,0) node{$P$};
\foreach \i in {1, 2, ..., 4} {
    \setcounter{Angle}{180-\i * 360 / 4 + 45}
    \vertex(a\i) at (\theAngle:1) [label=:]{};
    \vertex(b\i) at (\theAngle:1.3) [label=:]{};
}
\vertex(c1) at (135:1.6) [label=:]{};
\vertex(c2) at (45:1.6) [label=:]{};

\path
(a1) edge[color=purple,line width = \widthedgethick] (b2)
(a1) edge[color=purple,line width = \widthedgethick] (c2)
(b1) edge[color=purple,line width = \widthedgethick] (a2)
(b1) edge[color=purple,line width = \widthedgethick] (c2)
(c1) edge[color=purple,line width = \widthedgethick] (a2)
(c1) edge[color=purple,line width = \widthedgethick] (b2)

(a2) edge[color=purple,line width = \widthedgethick] (b3)
(b2) edge[color=purple,line width = \widthedgethick] (a3)
(c2) edge[color=purple,line width = \widthedgethick] (b3)
(c2) edge[color=purple,line width = \widthedgethick] (a3)

(a3) edge[color=purple,line width = \widthedgethick] (b4)
(b3) edge[color=purple,line width = \widthedgethick] (a4)

(a4) edge[color=purple,line width = \widthedgethick] (b1)
(a4) edge[color=purple,line width = \widthedgethick] (c1)
(b4) edge[color=purple,line width = \widthedgethick] (a1)
(b4) edge[color=purple,line width = \widthedgethick] (c1);

\end{scope}

\end{tikzpicture}
}

\newcommand{\AndrasfaiGraph}[2]{
\def\k{#2}
\pgfmathsetmacro\n{ 3*\k-1}
\pgfmathsetmacro\endk{ 2*\k-1}

\begin{tikzpicture}[x=#1 cm, y=#1 cm]
        \foreach \i in {1, 2,..., \n} {
            \setcounter{Angle}{180-\i * 360 / (3*\k-1)}
            \tinyvertex (v\i) at (\theAngle:1) []{};}

        \foreach \u in {1, 2,..., \n} {
            \foreach \j in {\k, ..., \endk} {
                
                \setcounter{Neighbour}{ \u+\j}
                \pgfmathparse{\theNeighbour< \n || \theNeighbour == \n ? 1 : 0}
                \ifnum\pgfmathresult>0
                    \path (v\u) edge[color=\col] (v\theNeighbour); 
                \fi
                
            }
        }
\end{tikzpicture}

}

\begin{document}

\author{
    Thomas Lesgourgues \thanks{Department of Combinatorics and Optimization, University of Waterloo, Canada. Email:  \texttt{tlesgourgues@uwaterloo.ca}}
\and Anita Liebenau \thanks{School of Mathematics and Statistics, UNSW Sydney, NSW 2052, Australia. 
  Email: \texttt{a.liebenau@unsw.edu.au} 
  This work was completed while the author was in residence at the Simons Laufer Mathematical Sciences Institute in
Berkeley, California, during the Spring 2025 semester, supported by the National Science Foundation under Grant No.\, DMS-1928930.}
\and Nye Taylor \thanks{School of Mathematics and Statistics, UNSW Sydney, NSW 2052, Australia. Email: \texttt{nye.taylor@student.unsw.edu.au}}
}

\sloppy

\title{Ramsey with purple edges}
\date{May 1, 2025}

\maketitle

\begin{abstract}
Motivated by a question of Angell, we investigate a variant of Ramsey numbers where some edges are coloured simultaneously red and blue, which we call purple. Specifically, we are interested in the largest number~$g=g(n;s,t)$, for some~$s$ and~$t$ and~$n<R(s,t)$, such that there exists a red/blue/purple colouring of~$K_n$ with~$g$ purple edges, with no red/purple copy of $K_s$ nor blue/purple copy of $K_t$. We determine~$g$ asymptotically for a large family of parameters, exhibiting strong dependencies with Ramsey--Tur\'{a}n numbers.
\end{abstract}

\section{Introduction}

Ramsey theory is one of the most active areas in combinatorics, with connections to extremal graph theory, number theory, discrete geometry, and theoretical computer science. The starting point of the area is Ramsey's Theorem~\cite{RamseyFormalLogic} which implies that for given natural numbers $s,t$ there exists an integer~$n$ such that every red/blue colouring of the edges of the complete graph on~$n$ vertices, denoted by~$K_n$, contains a red copy of~$K_s$ or a blue copy of~$K_t$. One of the most prominent open problems in the area is determining the smallest such~$n$, called the Ramsey number $R(s,t)$.

Erd\H{o}s and Szekeres~\cite{erdosszekeres1935} popularised Ramsey theory and showed that $R(t,t)\leq 4^{t}/\sqrt{t}$. A lower bound on \( R(t, t) \) of the form \( 2^{t/2 + o(1)} \), first established by Erd\H{o}s~\cite{erdos1947} as a prime example of the power of the probabilistic method, remains the best known to date -- having been improved only by a constant factor of two by Spencer~\cite{spencer1975ramsey}. 
The upper bound of the form $4^{t+o(1)}$ has been improved by various authors~\cite{Rodl1986Unpublished,Thomason1988,conlon2009new,sah2020diagonal}  culminating in the breakthrough result by Campos, Griffiths, Morris, and Sahasrabudhe~\cite{campos2023exponential}, who gave the first exponential improvement on $R(t,t)$. Their result has recently been improved by Gupta, Ndiaye, Norin, and Wei~\cite{gupta2024optimizing} who showed that for all sufficiently large $t$, we have $R(t,t)\leq 3.8^t$. 

More is known in the off-diagonal case, when~$s$ is a small constant and $t$ is large with respect to~$s$. Ajtai, Koml\'{o}s and Szemer\'edi~\cite{aks1980} proved that $R(s,t)=O_s(t^{s-1}/(\log t)^{s-2})$ for all $s,t$. For $s=3$, Shearer improved this upper bound to $R(3,t)\le (1+o(1))t^2/\log t$, which is now known to be tight up to a factor $(1/4+o(1))$ as established by Bohman and Keevash~\cite{bohman2013dynamic}, and independently by Fiz Pontiveros, Griffiths, and Morris~\cite{fpgm2013}, improving on earlier work by Kim~\cite{kim1995}. For $s=4$, Mattheus and Verstraete~\cite{mattheus2024asymptotics} showed recently that the upper bound by Ajtai, Koml\'{o}s and Szemer\'edi is tight up to a factor of $(\log t)^2$. 

In this article, we consider a natural variant inspired by a problem posed by Angell in {\em Parabola}~\cite{Parabola}, a mathematical magazine for high school students.

\begin{problem}[Angell~\cite{Parabola}]\label{DavidProblem}
    Consider $K_8$, the complete graph on the vertex set $\set*{1,\ldots,8}$, such that the edges $12$, $34$, $56$ and $78$ are joined by a double edge, one coloured red and one coloured blue. 
    Show that every colouring of the remaining (simple) edges using colours red and blue produces a red copy of $K_3$ or a blue copy of $K_4$.
\end{problem}

We invite the interested reader to find a solution before proceeding. Since $R(3,4)=9$, the addition of some double-coloured edges is necessary in this problem. We initially conjectured that a more general statement holds, namely that for any $s,t\ge 3$, every red/blue colouring of the edges of the complete graph on $R(s,t)-1$ vertices, with the addition of a perfect matching of double-coloured edges, contains a red copy of $K_s$ or a blue copy of $K_t$. 
This turns out to be false.
McKay and Radziszowski~\cite{McKayR45} employed computational methods to prove that $R(4,5)$ is equal to $25$, a fact that was recently also formally verified by Gauthier and Brown~\cite{gb2024-formalR45}. However, there exists a red/blue colouring of the edges of $K_{24}$ with a perfect matching of double-coloured edges, that contains neither a red copy of $K_4$ nor a blue copy of $K_5$. We outline the computational methods used to obtain this result and other small cases in~\cref{sec:comp}.

Beyond the consideration of small cases and of matchings, we find the presence of edges that are both red and blue to be of broader interest. We prefer to view double-coloured edges as edges of a third colour, which we shall henceforth call {\em purple}. A {\em red/blue/purple colouring} of $K_n$ is a partition of the edges of $K_n$ into three graphs {\RBP}. 
Problem~\ref{DavidProblem}, for example, may be rephrased as showing that any red/blue/purple colouring \RBP{} of $K_8$, where $P$ forms a perfect matching, 
contains a red/purple triangle or a blue/purple copy of $K_4$, meaning  $K_3\subseteq R\cup P$ or $K_4\subseteq B\cup P$. 

One possibility to approach this variant is by considering $f(n,m)$, which is the size of the largest red/purple or blue/purple clique guaranteed to exist in every red/blue/purple colouring of $K_n$ with exactly $m$ purple edges.
The Erd\H{o}s--Szekeres upper bound on $R(t,t)$ implies that $f(n,m)$ is at least $(\log n)/2$ for every $m$, and a random colouring establishes that $f(n,m)=O(\log n)$ when $m\le cn^2/2$ for every constant $0<c<1$. 
Nagy~\cite{nagy2016density} improved the lower bound by adapting an elegant argument of Erd\H{o}s and Szemer\'{e}di~\cite{ErSz1972ramsey} on a different variant on Ramsey numbers. 

In this article, we explore the problem from the following 
perspective.
We say that a graph $G$ is {\em $H$-free} if~$G$ does not contain a copy of $H$ as a subgraph. 
Given integers $s,t$, we say that a red/blue/purple colouring \RBP{} is {\em $(s,t)$-free} if $P\cup R$ is~$K_s$-free and $P\cup B$ is~$K_t$-free. 

\begin{definition}
    \label{def:PurpleRamsey}
    Given integers $s,t\geq 2$ and $n<R(s,t)$, let  $g=g(n;s,t)$ denote the largest integer such that there exists an $(s,t)$-free red/blue/purple colouring \RBP{} of $K_n$ that satisfies $|P| = g$. 
\end{definition}
Clearly, for $n\geq R(s,t)$ no such colouring exists, whereas for $n < R(s,t)$ at least one colouring, possibly with zero purple edges, does exist. So, $g$ is well-defined. For simplicity, we extend~\cref{def:PurpleRamsey} to non-integers $s$ and $t$, as the largest integer such that there exists a red/blue/purple colouring \RBP{} of $K_n$ that satisfies $\omega(R\cup P)<s$, $\omega(B\cup P)<t$ and $|P| = g$, where $\omega(G)$ denotes the clique number of a graph~$G$ as usual. That is for non-integers~$s,t$ we let $g(n;s,t):=g(n;\ceil{s},\ceil{t})$.

It turns out that $g(n;s,t)$ is closely related to the so-called Ramsey-Tur\'an numbers. For integers~$n$,~$s$ and~$t$, the Ramsey-Tur\'an number \RT{s}{t}{n} is the maximum number of edges in an~$n$-vertex $K_s$-free graph $G$ with independence number $\alpha(G)< t$.  First considered by Andr{\'a}sfai~\cite{andrasfai1962extremalproblem} in the triangle case, Ramsey-Tur\'an numbers were systematically introduced by S\'os, who proved many deep results throughout her career (e.g.~\cite{EMST1972someI,EMST1972someII,EMST1972someIII,EHSSS1993turan,EHSSS1994turan}). It is not hard to see that Ramsey-Tur\'an numbers yield natural upper bounds on~$g(n;s,t)$. Indeed, let \RBP{} be a red/blue/purple colouring of $K_n$ such that $R\cup P$ is $K_s$-free, $B\cup P$ is $K_t$-free, and $|P|$ is maximal. Then $R\cup P$ has independence number $\alpha(R\cup P)=\omega(B)\leq \omega(B\cup P)< t$. It follows that
\begin{equation}\label{eq:upper_bound_RT}
    g(n;s,t)=|P|\leq |P\cup R|\leq \RT{s}{t}{n}.
\end{equation}

We refer to the excellent survey by Simonovits and S\'os~\cite{simonovits2001survey} on Ramsey-Tur\'an numbers, and summarise here only the results that are most important for us. We first note that 
$\RT{s}{cn}{n}$ is known precisely for every $c\geq 1/(s-1)$ due to Tur\'{a}n's theorem~\cite{turan1941extremal}, see~\cref{sec:RT_Theory}. Many classical theorems in the area yield bounds on $\RT{s}{cn}{n}$ when $c$ is a sufficiently small constant. 
Erd\H{o}s and S\'os~\cite{erdos1970someremarks} studied this specific problem for odd cliques $K_s$, and showed that $\RT{s}{cn}{n} = ((s-3)/(s-1)+O(c))n^2/2$. 
Brandt~\cite{brandt2010triangle} gave asymptotically tight bounds on $\RT{3}{cn}{n}$ when~$c\in (0,1/3)$, and L{\"u}ders and Reiher~\cite{luders2019ramsey} extended this result to all odd integers~$s$ and small enough~$c$.
The even clique cases proved to be much more difficult. 
The first result in that direction was obtained by Szemer{\'e}di~\cite{szemeredi1972graphs} who proved $\RT{4}{cn}{n}\leq (1/8+ O(c))n^2$ for small enough~$c$, using a version of his seminal Regularity Lemma. An ingenious construction by Bollob{\'a}s and Erd\H{o}s~\cite{bollobas1976ramsey} shows that this bound is best possible up to the dependency in~$c$. 
Erd\H{o}s, Hajnal, S\'os, and Szemer{\'e}di~\cite{erdos1983more} used the Bollob{\'a}s-Erd\H{o}s construction and showed that $\RT{s}{cn}{n}= ((3s-10)/(3s-4)+O(c))n^2/2$ for all even integers~$s$. 
Recently, Fox, Loh, and Zhao~\cite{fox2015critical} 
refined the dependency on $c$ when $s=4$, and L{\"u}ders and Reiher~\cite{luders2019ramsey} proved that $\RT{s}{cn}{n}= ((3s-10)/(3s-4)+c-c^2+o(1))n^2/2$ for all even integers~$s$.

Our first result shows that the trivial upper bound~\cref{eq:upper_bound_RT} is asymptotically correct in that range, that is when $t=cn$ for sufficiently small $c$. 
\begin{thm}\label{thm:GeneralSmallLinear}
For any integer $s\geq 3$, and for any $c>0$ small enough, we have
    \[g(n;s,cn) = (1-o(1))\; \RT{s}{cn}{n}.\]
\end{thm} 

Our next theorem centres on triangles, where we explore the behaviour of $g(n;3,t)$ for a broader range of~$t$. Let us first recall what is known about $\RT{s}{cn}{n}$ when $s=3$. It follows from Mantel's theorem that $\RT{3}{cn}{n}=\floor{n^2/4}$ for any $c\geq 1/2$, while Brandt's work~\cite{brandt2010triangle} covers the range $c<1/3$ asymptotically. In the intermediate range, Andr{\'a}sfai proved that certain blow-ups of $C_5$ are extremal constructions for $\RT{3}{cn}{n}$ when $c\in[2/5, 1/2]$, and conjectured that similar constructions should provide extremal examples for the remaining range of $c$. {\L}uczak, Polcyn, and Reiher~\cite{luczak2022ramsey,luczak2025next} confirmed this conjecture for $c\in[4/11,2/5]$ and recently announced a proof of the conjecture. We refer the reader to~\cref{sec:RT_Theory}, for a detailed presentation of these results and the Andr{\'a}sfai Conjecture. Using these, we give an almost complete picture of $g(n;3;t)$, showing first that~\cref{eq:upper_bound_RT} is asymptotically correct when~$t=cn$ for every~$c$ for which the Andr{\'a}sfai Conjecture holds. We also bound $g(n;3,t)$ in terms of Ramsey-Tur\'{a}n numbers when~$t$ is sublinear.

\begin{thm}\label{thm:TriangleCase}\
\begin{enumerate}
    \item For all $c\in \left(0,1\right]\setminus \left(\frac{1}{3},\frac{4}{11}\right)$ we have \label{item:TriangleLinearCase}
    \begin{equation}
        g(n;3,cn) = (1-o(1)) \; \RT{3}{cn}{n}.
    \end{equation}
    Moreover, assuming the Andr\'asfai Conjecture, we obtain that the same result holds also for $c\in \left(\frac{1}{3},\frac{4}{11}\right)$.
    \item For every $\varepsilon>0$ there exists $\delta>0$ such that for sufficiently large~$n$ and~$t=(1+\varepsilon)\sqrt{2n\log n}$, we have
    \[g(n;3,t)\geq\delta\cdot \RT{3}{t}{n}.\]
    Moreover, if $t\gg \sqrt{n\log n}$, we have
    \[g(n;3,t)\geq \Big({\frac{1}{2}-o(1)}\Big)\; \RT{3}{t}{n}. \]\label{item:TriangleSublinearCase}
\end{enumerate}
\end{thm}

We remark that, even without assuming the Andr\'asfai Conjecture,~\cref{thm:TriangleCase}\ref{item:TriangleLinearCase} for $c=1/3$ implies that $g(n;3,c'n)$ is at least $0.91\, \RT{3}{c' n}{n}$ for $c'$ in the gap range $c'\in \left(1/3,4/11\right)$. 

Next we note that the lower bound on $t$ in~\cref{thm:TriangleCase}\ref{item:TriangleSublinearCase} is essentially best possible. Indeed, the current, best known bounds on $R(3,t)$ are 
\begin{equation}\label{eq:bounds-R3t}
\left(\frac14-o(1)\right)\, \frac{t^2}{\log t} \le R(3,t)\le (1+o(1))\, \frac{t^2}{\log t },  
\end{equation}
due to Bohman and Keevash~\cite{bohman2013dynamic},  Fiz Pontiveros, Griffiths, and Morris~\cite{fpgm2013}, and Shearer~\cite{shearer1983}. The upper bound implies that $g(n;3,t)$ is undefined if $t \le \gamma \sqrt{n\log n}$, for any $\gamma <  1/\sqrt{2}$. Moreover, the lower bound in~\eqref{eq:bounds-R3t} is conjectured to be asymptotically tight in both~\cite{bohman2013dynamic,fpgm2013}. If true, this would imply that $g(n;3,t)$ is undefined if $t \le \gamma \sqrt{n\log n}$, for any $\gamma <  \sqrt{2}$, that is, \cref{thm:TriangleCase}\ref{item:TriangleSublinearCase} would cover practically the full range of~$t$. 

Finally, we observe that we cannot expect the constant $\delta$ in~\cref{thm:TriangleCase}\ref{item:TriangleSublinearCase} to be close to~1 when~$t$ is within a constant factor of $\sqrt{n\log n}$. Indeed, we have the following. 

\begin{observation}\label{obs:ADMR}
    Let $\gamma$ and $\delta$ be real numbers satisfying $0<\delta < 1/(2\gamma^2) < 1$. Then there exists an integer $n_0$ such that for all $n\ge n_0$ and $t= \gamma \sqrt{n\log n}$ we have that 
    \(g(n;3,t) < (1-\delta)\cdot \RT{3}{t}{n},\) 
    assuming $g(n;3,t)$ is well-defined. 
\end{observation}

We would like to thank Anastos, Das, Morris and Rathke~\cite{AnastosPrivComm} for pointing this fact out to us, which is a simple consequence of Shearer's proof of the lower bound in~\eqref{eq:bounds-R3t}. We provide the details in~\cref{sec:TriangleSubLinear}.

To discuss smaller ranges of $t$, define the {\em inverse Ramsey number}, denoted by $R^{-1}(s,n)$, to be the minimum independence number of a $K_s$-free graph on~$n$ vertices. Observe that $g(n; s,t)$ is defined if and only if $R^{-1}(s,n)<t$. Using a probabilistic argument, we give a general lower bound on $g(n;s,t)$ for $t$ almost as small as $R^{-1}(s,n)$.

\begin{theorem}\label{thm:ProbabilisticLowerBound}
    There exists an absolute constant $p>0$ such that for any fixed $s\geq 3$, for~$n$ large enough and any $t\in\brackets*{R^{-1}(s,n)\log n,n}$, we have
    \[p\cdot\RT{s}{t/\log n}{n}\leq g(n;s,t).\]
\end{theorem}

To compare this general lower bound with the upper bound in~\eqref{eq:upper_bound_RT}, it is useful to understand the behaviour of the Ramsey-Turán numbers $\RT{s}{t}{n}$ and $\RT{s}{t'}{n}$ when $t$ and $t'$ are  not too far apart.
For example, Theorem~1.10 in~\cite{fox2015critical} implies that $\RT{4}{t}{n}\ge n^2/8+o(n^2)$ if $t= n/(\log n)^K$, for any constant $K>1$. Thus, \cref{thm:ProbabilisticLowerBound} implies that for $t= n/(\log n)^{K-1}$ we have that $g(n;s,t)=\Theta(n^2)$. Sudakov proved that if $t= e^{-\omega((\log n)^{1/2})}n$ then $\RT{4}{t}{n}=o(n^2)$, from which we infer via~\eqref{eq:upper_bound_RT} that $g(n;4,t) =o(n^2)$ for such $t$. More generally, Balogh, Hu, and Simonovits~\cite{balogh2015phase} showed that $\RT{s}{t}{n}=o(n^2)$ when $t=R^{-1}(s,n)n^{\varepsilon}$ for some $\varepsilon\leq 1/s^2$. They also gave a more precise range of~$t$ where a transition from $\RT{s}{t}{n}=\Theta(n^2)$ to $\RT{s}{t}{n}=o(n^2)$ occurs, assuming the 
following long-standing conjecture on Ramsey numbers. 
\begin{conjecture}\label{conj:OnRamsey}
    For every integer $s\geq 3$, there exist $\vartheta=\vartheta(s)>0$ such that for any $t$ large enough
    \[R(s-1,t)t^\vartheta\leq R(s,t).\]
\end{conjecture}
Assuming this conjecture, it is proved in~\cite{balogh2015phase} that  $\RT{s}{R^{-1}(\ell,n)}{n}$ is quadratic in~$n$ if $\ell\leq\frac{s+1}{2}$ and sub-quadratic if $\ell> \frac{s+1}{2}$. Using this result and our probabilistic bound from~\cref{thm:ProbabilisticLowerBound}, we derive the following general behaviour of $g(n;s,t)$.
\begin{theorem}\label{prop:Generalbehaviour_precise}
    For any $s\geq 3$, assuming~\cref{conj:OnRamsey},
    \begin{enumerate}
        \item if $t\geq R^{-1}\parens*{\frac{s+1}{2},n}\log n$, then $g(n;s,t)=\Theta(n^2)$.
        \item if $t\leq R^{-1}\parens*{\frac{s+3}{2},n}$, then $g(n;s,t)=o(n^2)$.
    \end{enumerate}
\end{theorem}

{\bf Organisation.}
In~\cref{sec:Preliminaries} we introduce notation and the key tool that will be used throughout the article.~\cref{sec:TriangleLinearSmall} is devoted to the proof of~\cref{thm:TriangleCase}\ref{item:TriangleLinearCase}.~\cref{sec:TriangleSubLinear} concludes the triangle case, proving~\cref{thm:TriangleCase}\ref{item:TriangleSublinearCase}. In~\cref{sec:GeneralS}, we present our results for general~$s$ proving~\cref{thm:GeneralSmallLinear,thm:ProbabilisticLowerBound}. We conclude the article with some computational results in~\cref{sec:comp} and discussions and open problems in~\cref{sec:Conclusion}.

\section{Preliminaries}\label{sec:Preliminaries}

{\bf Notation.} 
We use standard graph theoretic notation. In particular, given a graph~$G$, we write~$V(G)$ for its vertex set and~$E(G)$ for its edge set, with~$v(G)=|V(G)|$ and~$e(G)=|E(G)|$. We further write~$\alpha(G)$, $\omega(G)$, and~$\delta(G)$ for the independence number, for the clique number, and for the minimum degree  of $G$, respectively. Given graphs~$F$ and~$G$, we say that~$G$ is~$F$-free if~$G$ does not contain a subgraph that is  isomorphic to~$F$. 
A {\em blow-up} of an $n$-vertex graph $G$ is a graph obtained by replacing each vertex $i$ of $G$ by an independent set $V_i$, and each edge $ij$ of $G$ by a complete bipartite graph between the parts $V_i$ and $V_j$. A blow-up is called {\em equitable} if $\big| |V_i|-|V_j|\big|\le 1$ for all $i,j$. 

We often abuse notation identifying a graph with its edge set. For example, we write~$\alpha(R)$ for the independence number of the spanning subgraph of~$K_n$ induced on the edges of~$R$, 
for a $3$-colouring \RBP{} of~$K_n$, and similarly for the clique number $\omega(R)$. Furthermore, for a subset $S\subseteq V(K_n)$, we write~$R[S]$ for the subgraph induced on the vertices of~$S$ and the edges in~$R$.

Unless stated otherwise, logarithms are in base~$2$. We use the standard Landau symbols $O$, $\Omega$, $\Theta$, and $o$ to denote the asymptotic behaviour of functions with respect to~$n$, the number of vertices, tending to $\infty$, as well as $f(n)\gg g(n)$ to denote $g(n)=o(f(n))$. We say that an event $A(n)$ holds with high probability if~$\lim_{n\to\infty}\Pr[A(n)] = 1$.

\subsection{Blow-up colouring}

We start with a general construction of a red/blue/purple colouring \RBP{} of $K_n$ that we will use throughout the article. 

\begin{definition}[Blow-up colouring]\label{def:PurpleBlowUp}
    Let $n\ge k \ge 1$ be integers and let~$G$ be a graph on~$k$ vertices $v_1,\ldots,v_k$. The {\em $n$-blow-up colouring} of~$G$ is a colouring \RBP{} of~$K_{n}$ defined as follows.
    Let $H$ be an equitable blow-up of $G$ created by replacing every vertex $v_i\in V(G)$ by a set of vertices $V_i$ of size either $\beta^{-}=\floor{\frac{n}{k}}$ or $\beta^{+}=\ceil{\frac{n}{k}}$, such that $\sum_i |V_i|=n$, and replacing every edge $v_iv_j$ of $G$ by a complete bipartite graph between the sets $V_i$ and $V_j$. For each $i\in[k]$, denote the vertices of $V_i$ by $\set{u_1^i,\ldots,u_{\beta^{-}}^i}$ or $\set{u_1^i,\ldots,u_{\beta^{+}}^i}$. For every edge $v_iv_j$ of $G$, and for every $\ell\in[\beta^{+}]$, colour the edge $u_{\ell}^iu_{\ell}^j$ of $H$, if it exists, red. Colour all other edges of $H$ purple. Finally, colour all edges in the complement of $H$ blue. Let $R$ be the set of all red edges, $B$ the set of all blue edges, and $P$ the set of all purple edges.
\end{definition}

An illustration of a $10$-blow-up colouring of the cycle graph $C_4$ can be found in~\cref{fig:PurleBlowUp}. We remark that $n$-blow-up colourings are not uniquely defined, as their definition is up to the choice of vertices being replaced by sets of size $\beta^+$ or $\beta^-$. However, in every $n$-blow-up colouring of a graph $G$, the red edges form $\beta^{-}$ vertex-disjoint copies of $G$, and possibly a vertex-disjoint copy of an induced subgraph of $G$. 
\begin{figure}[ht]
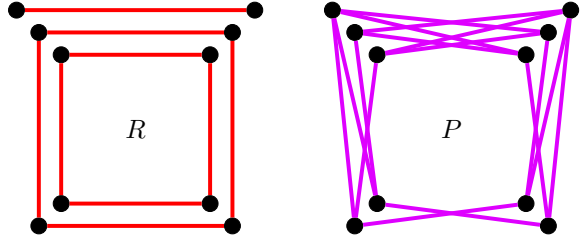

    \centering
    \ExamplePurpleBlowUp{1.4}
    \caption{This is an example of red and purple edges in a 10-blow-up colouring of $C_4$.}
    \label{fig:PurleBlowUp}
\end{figure}

The following proposition captures the core properties of blow-up colourings.

\begin{proposition}\label{thm:PropPurpleBlowUp}
    Let $G$ be a graph on~$k$ vertices and $n\ge k$ be an integer. Each $n$-blow-up colouring $\RBP$ of $G$ satisfies 
    \begin{enumerate}
        \item $\omega(R\cup P)=\omega(G)$, \label{item:RPClique}
        \item $\alpha(R)=\omega(B\cup P)\leq \ceil*{n/k}\, \alpha(G)$,\label{item:BPClique}
        \item $\big| |P| - e(G)\, n^2/k^2 \big| \le e(G)3n/k,$ and \label{item:SizeP}
        \item $|R|\le e(G)\, (n/k+1)$. \label{item:SizeR} 
    \end{enumerate}
\end{proposition}

\begin{proof}
Let \RBP{} be an $n$-blow-up colouring of $G$. Trivially, any edge in $R\cup P$ corresponds to an edge in~$G$, while for any two vertices $u,v \in V_i$ for some~$i$, we have $uv\not\in R\cup P$. Therefore $\omega(R\cup P)=\omega(G)$, verifying~\cref{item:RPClique}.
The red graph $R$ consists of $\ceil*{\frac{n}{k}}$ vertex-disjoint copies of induced subgraphs of~$G$ that cover all~$n$ vertices. Therefore, $\alpha(R)\le \alpha(G)\, \ceil*{\frac{n}{k}}$, proving~\cref{item:BPClique}. 
To verify~\cref{item:SizeP,item:SizeR}, let $m=\floor*{\frac{n}{k}}$ and 
note that 
\[R^-\subseteq R\subseteq R^+,\qquad P^-\subseteq P\subseteq P^+,\]
where $R^-\cup B^- \cup P^-$ is an $(mk)$-blow-up colouring of $G$, and where $R^+\cup B^+ \cup P^+$ is an $((m+1)k)$-blow-up colouring of $G$. We then obtain
\[e(G)(m^2-m)
\leq|P| \leq 
e(G)\big((m+1)^2-(m+1)\big),\]
hence 
\[e(G)\Big(\frac{n^2}{k^2}-\frac{3n}{k}+1\Big)
\leq|P| \leq 
e(G)\Big(\frac{n^2}{k^2}+\frac{n}{k}+1\Big),\]
and
\[\size*{ |P| - e(G)\, \frac{n^2}{k^2} } \le e(G)\frac{3n}{k}.\qedhere\]
\end{proof}

\subsection{Ramsey-Tur\'{a}n theory}\label{sec:RT_Theory}

In this subsection, we summarise results on Ramsey-Tur\'an numbers that we utilise in this paper. We start with Tur\'an's Theorem. 
For non-negative integers~$s$ and~$n$, denote by~$T_{n,s}$ the equitable blow-up of $K_{s}$ on $n$ vertices. 

\begin{theorem}[Tur\'an~\cite{turan1941extremal}]\label{thm:Turan}
    For $s\geq 3$, the maximal number of edges among all $n$-vertex $K_s$-free graphs is obtained uniquely by $T_{n,s-1}$. 
\end{theorem}

Since the independence number of $T_{n,s-1}$ is equal to $\lceil n/(s-1)\rceil$, we obtain immediately that 
\begin{align}\label{eq:Turan-edges}
    \RT{s}{t}{n} &= e(T_{n,s-1}) = \frac12\Big(1-\frac{1}{s-1}\Big)n^2+O(n)
\end{align}
for all $t> n/(s-1)$. 
It is clear that the red/purple subgraph $R\cup P$ of an $n$-blow-up colouring of $K_{s-1}$ is exactly the Tur\'an graph $T_{n,s-1}$ and therefore, \cref{thm:PropPurpleBlowUp} implies that $g(n;s,t)$ is asymptotically equal to \RT{s}{t}{n} when $t>n/(s-1)$.

\begin{corollary}\label{thm:LinearLargeC}
    For any $s\geq 3$, and $c>1/(s-1)$, we have
    \[g(n;s,cn) = \big(1-O(1/n)\big)\, \RT{s}{cn}{n}.\]
\end{corollary}

\begin{proof}
Let $\RBP$  be an $n$-blow-up colouring of $K_{s-1}$ for some large enough~$n$. The graph $R\cup P$ is the Tur\'an graph $T_{n,s-1}$ which is $K_s$-free and has independence number $n/(s-1)< cn$. By~\cref{thm:PropPurpleBlowUp} we obtain  
\[|P|= \binom{s-1}{2}\parens*{\frac{n^2}{(s-1)^2}-O\Big(\frac{n}{s-1}\Big)} =\frac12\Big(1-\frac{1}{s-1}\Big)n^2+O(n),\]
and~\cref{eq:Turan-edges} proves the claim. 
\end{proof} 

The interesting range for Ramsey-Tur\'an numbers is therefore when $t\le 1/(s-1)$. Let us focus on triangles first, that is when $s=3$. Note that the neighbourhood of any vertex in a triangle-free graph is an independent set, therefore for all integers $t\leq n$ we have the trivial bound 
\begin{equation}\label{eq:trivial_bound}
    \RT{3}{t}{n} \leq \frac{n(t-1)}{2}.
\end{equation}

This upper bound is tight if there exists a triangle-free $(t-1)$-regular graph on $n$ vertices with independence number equal to $t-1$. Let $\Omega$ to be the set of all pairs of natural numbers $(d,k)$ such that there exists a triangle-free, $d$-regular graph on~$k$ vertices with independence number $d$. 
Observe that if $(d,k)\in\Omega$ then $(ad,ak)\in \Omega$ for any positive integer $a$, since an equitable blow-up of a graph certifying $(d,k)\in\Omega$ is an $ad$-regular graph on~$ak$ vertices with independence number $ad$.
Therefore, the following theorem by Brandt covers the range of $t=cn$  for $0<c< 1/3$ asymptotically .
\begin{theorem}[Brandt~\cite{brandt2010triangle}]\label{thm:RTBrandt}
The set $\set*{\frac{d}{k}\colon (d,k)\in\Omega}\cap\parens*{0,\frac13}$ is dense in $\parens*{0,\frac13}$.
\end{theorem}
It follows from~\cref{thm:RTBrandt} that~\cref{eq:trivial_bound} is asymptotically tight for any positive real number $c<1/3$, that is $\RT{3}{cn}{n}=(c+o(1))n^2/2$. Somewhat surprisingly, it turns out that this is not true in general for the remaining range of $c$. To discuss this case, we need to introduce the following class of graphs. For $k\geq 2$, let $\Gamma_k$ denote the Cayley graph over $\mathbb{Z}/(3k-1)\mathbb{Z}$  generated by the sum-free set $\set{k,k+1\ldots,2k-1}$.  That is, for every $k\ge 2$, the vertex set of $\Gamma_k$ is $\mathbb{Z}/(3k-1)\mathbb{Z}$ and its edge set is 
$$E(\Gamma_k) = \left\{ij \in \binom{V(\Gamma_k)}{2}\, : \, i-j \in \{k,k+1\ldots,2k-1\}\right\}.$$
The graphs $\Gamma_k$ are now called Andr\'asfai graphs, see~\cref{fig:AndrasfaiGraphs} for some small examples. 
\def\sizeAndrasfai{1} 
\def\col{teal} 
\begin{figure}[H]
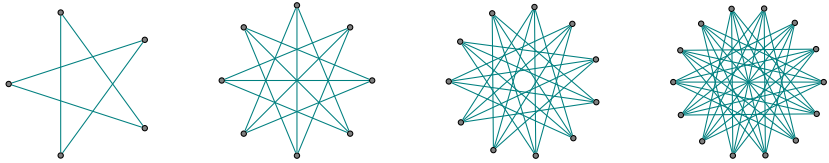

    \centering
    \subfloat{{\AndrasfaiGraph{\sizeAndrasfai}{2}}}
    \quad    
    \subfloat{{\AndrasfaiGraph{\sizeAndrasfai}{3}}}
    \quad
    \subfloat{{\AndrasfaiGraph{\sizeAndrasfai}{4}}}
    \quad
    \subfloat{{\AndrasfaiGraph{\sizeAndrasfai}{5}}}
    \caption{ Andr\'asfai graphs $\Gamma_2$, $\Gamma_3$, $\Gamma_4$, and $\Gamma_5$.}%
    \label{fig:AndrasfaiGraphs}%
\end{figure}

Importantly, Andr\'asfai graphs are triangle-free $k$-regular graphs on $3k-1$ vertices with independence number $\alpha(\Gamma_k)=k$. We refer to~\cite{luczak2022ramsey, luczak2025next} for more background on Andr\'asfai graphs. 

For $kn/(3k-1)\le t < (k-1)n/(3k-4)$, we define, following~\cite{luczak2025next}, the ``canonical'' blow-up $\Gamma(n; k,t)$ of $\Gamma_k$ by replacing vertices $1$, $k$, $2k$ in $\Gamma_k$ with sets of $(k-1)n-(3k-4)t$ vertices each, replacing the remaining $3k-4$ vertices in $\Gamma_k$ with sets of $3t-n$ vertices, and replacing every edge $uv$ of $\Gamma_k$ by a complete bipartite graph. See, for example, Figure~1.2 in~\cite{luczak2025next} for an illustration. It is not hard to see that~$\Gamma(n;k,t)$ is a blow-up of $\Gamma_k$ on~$n$ vertices such that $\alpha(\Gamma(n;k,t))=t$ and
\begin{equation}\label{eq:SizeCanonicalBlowUp}
e(\Gamma(n;k,t))=\frac{1}{2} k(k-1)n^2- k(3k-4)tn+\frac{1}{2}(3k-4)(3k-1)t^2. 
\end{equation}

It can also be shown that the canonical blow-up $\Gamma(n;k,t)$ maximises the number of edges among all $n$-vertex blow-ups of $\Gamma_k$ that have independence number at most $t$, see for example Section 3 in~\cite{luczak2021andrasfai} for a detailed discussion of the properties of $\Gamma(n;k,t)$.  
Andr\'asfai~\cite{andrasfai1962extremalproblem} proved that for $2n/5 \le t \le n/2 $ the graph $\Gamma(n;2,s)$ maximises the number of edges of any triangle-free graph with independence number at most $t$, that is $\RT{3}{t+1}{n} = e(\Gamma(n;2,t))$ for such $t$. We remark here that Ramsey-Tur\'{a}n numbers are defined inconsistently in the literature (using $\alpha\le t$ or $\alpha< t$), hence the appearance of a ``$+1$'' is not present in~\cite{andrasfai1962extremalproblem}, for example. Andr\'asfai then conjectured that the graphs~$\Gamma_k$ yield extremal graphs for the Ramsey-Tur\'an numbers $\RT{3}{t}{n}$ for all $ n/3 < t \le n/2$. 
\begin{conjecture}[Andr\'asfai~\cite{andrasfai1962extremalproblem}]\label{conj:Andrasfai}
    For all integers~$t$ and~$n$ we have 
    $\RT{3}{t+1}{n} = e(\Gamma(n;k,t))$, where~$k$ is the unique integer such that $kn/(3k-1)\le t< (k-1)n/(3k-4)$.
\end{conjecture}

As mentioned above, the conjecture was proved for $k=2$ by Andr\'asfai~\cite{andrasfai1962extremalproblem}. After being dormant for about 50 years, the conjecture was recently attacked in a series of papers by {\L}uczak, Polcyn, and Reiher. They proved the conjecture for $k=3$ in~\cite{luczak2022ramsey} and for $k=4$ in~\cite{luczak2025next}. 
Thus, we may assume that 
\begin{align}  
 &\RT{3}{t+1}{n} = e(\Gamma(n;k,t)) \  \text{ for all } \ 4n/11 \le t \le n/2, 
\end{align}
where $k\in\{2,3,4\}$ is the unique integer satisfying $kn/(3k-1)\le t\le (k-1)n/(3k-4)$. 
{\L}uczak, Polcyn and Reiher~\cite{luczak2021andrasfai} also proved the conjecture for all integers~$k$ and all~$t$ which are sufficiently close to the critical values $kn/(3k-1)$. Even more excitingly, they recently announced a proof for all~$k$ and all admissible~$t$ and~$n$ in~\cite{luczak2024strong}.

For general $s\geq 4$, as mentioned in introduction, the work of Erd\H{o}s and S\'os~\cite{erdos1970someremarks}, Szemer{\'e}di~\cite{szemeredi1972graphs}, Bollob{\'a}s and Erd\H{o}s~\cite{bollobas1976ramsey}, Erd\H{o}s, Hajnal, S\'os, and Szemer{\'e}di~\cite{erdos1983more}, and Fox, Loh, and Zhao~\cite{fox2015critical}, culminated in the following theorem due to L{\"u}ders and Reiher~\cite{luders2019ramsey}.

\begin{theorem}[L{\"u}ders and Reiher~\cite{luders2019ramsey}]\label{thm:LudersReiher}
For all $\ell\geq 2$, and $c=o(1/\ell)$,     
\begin{align*}
\RT{2\ell+1}{cn}{n}&= \Big(\frac{\ell-1}{\ell}+c+o(1)\Big)\frac{n^2}{2},\\
\RT{2\ell}{cn}{n}&= \Big(\frac{3\ell-5}{3\ell-2}+c-c^2+o(1)\Big)\frac{n^2}{2}.    
\end{align*}
\end{theorem}

\section{The triangle case with linear independence number}\label{sec:TriangleLinearSmall}

We start our analysis with a focus on the case when $s=3$ and $t$ is linear in~$n$, hence proving~\cref{thm:TriangleCase}\ref{item:TriangleLinearCase}. Recall that \cref{thm:RTBrandt} implies that $\RT{3}{cn}{n}=(c+o(1))n^2/2$ for  $c<1/3$. We first show that $g(n;3,cn)$ is also asymptotically equal to $cn^2/2$ for $c<1/3$. Note that $\RT{3}{cn}{n}\le cn^2/2$ is the trivial bound in~\cref{eq:trivial_bound}. 

\begin{theorem}\label{thm:TriangleLinearSmall}
    For every constant $0<c\leq 1/3$, we have
    \[g(n;3,cn) \geq \frac{n^2}{2}(c-o(1)).\] 
\end{theorem}

\begin{proof}
     Fix $\eta>0$. Let $d,k$ be integers such that $c-\eta< \frac{d}{k}< c$, and such that there exists a triangle-free $d$-regular graph $G$ on~$k$ vertices with independence number $d$, as guaranteed by~\cref{thm:RTBrandt}. For~$n$ large enough, let $\RBP$ be an $n$-blow-up colouring of $G$. It follows from~\cref{thm:PropPurpleBlowUp} that $R\cup P$ is a triangle-free graph and, since~$n$ is large enough, $R$ has independence number at most
     \[
        \ceil*{\frac{n}{k}}\alpha(G)
        \leq \Big(\frac{d}{k}+\frac{d}{n}\Big)n
        < cn.
     \]
     We additionally obtain from~\cref{thm:PropPurpleBlowUp} that $\abs*{\size{P}-e(G)n^2/k^2}\leq 3e(G)n/k$, and thus, for~$n$ large enough, 
     \begin{align*}
         \size{P}\geq e(G)\Big(\frac{n^2}{k^2}-\frac{3n}{k}\Big) = \frac{n^2}{2}\Big(\frac{d}{k}-\frac{3d}{n}\Big) > \frac{n^2}{2}(c-\eta).
     \end{align*}
     Since $\eta>0$ can be chosen arbitrarily small, we obtain the desired result.
\end{proof}

Recall that \cref{conj:Andrasfai} asserts that for every integer $t\in(n/3,n/2)$, canonical blow-ups of Andr{\'a}sfai graphs are extremal for $\RT{s}{t}{n}$. We now prove that these graphs yield a lower bound for $g(n;3,t)$ up to some linear error term.

\begin{theorem}\label{thm:TriangleIntermdC}
For every $c\in\parens*{\frac13,\frac12}$ and positive integer $n$, we have
    \[g(n;3,t+1) \geq e(\Gamma(n;k,t))-O(n),\]
    where $t=\ceil{cn}$ and~$k$ is the unique integer such that $kn/(3k-1)\le t< (k-1)n/(3k-4)$.    
\end{theorem}

\begin{proof}
Fix $c\in(1/3,1/2)$, and let $k\geq 2$ be the unique integer such that $k/(3k-1)\leq c < (k-1)/(3k-4)$. Let~$n$ be large enough, in particular such that $t:=\ceil{cn}$ satisfies $kn/(3k-1)\leq t < (k-1)n/(3k-4)$. Let $\Gamma(n;k,t)$ be the canonical blow-up of the Andr{\'a}sfai graph $\Gamma_k$ defined in~\cref{sec:RT_Theory}. Label the vertices of $\Gamma(n;k,t)$ as 
\begin{align*}
V(\Gamma(n;k,t)) = 
&\big\{\{u_1^i,\ldots,u_{3t-n}^i\}\colon i\in[3k-1]\setminus\set{1,k,2k}\big\}\\
\cup&
\big\{\{u_1^i,\ldots,u_{(k-1)n-(3k-4)t}^i\}\colon i\in\set{1,k,2k}\big\}.    
\end{align*}

We note that the canonical blow-up is composed of $(3k-4)$ ``large'' sets of $(3t-n)$ vertices, and 3 sets of size $(k-1)n-(3k-4)t$.
We define an edge colouring \RBP{} of $K_n$ in an analogous fashion to the blow-up colourings from~\cref{def:PurpleBlowUp}, but using the canonical blow-up $\Gamma(n;k,t)$ instead of an equitable one. For any $i,j\in[3k-1]$ and any $\ell\in[3t-n]$, if $u_i^\ell u_j^\ell$ is an edge of $\Gamma(n;k,t)$, then colour that edge red. Colour all other edges of~$\Gamma(n;k,t)$ in purple. Colour all edges not in $\Gamma(n;k,t)$ in blue. Similarly to blow-up colourings, the red edges form $(3t-n)$ vertex-disjoint copies of either $\Gamma_k$ or a subgraph of $\Gamma_k$, hence $|R|\leq (3t-n)e(\Gamma_k)\leq 3k^2n$.

Observe that $R\cup P$ is a triangle-free graph since $\Gamma_k$ is. We now study the independence number of $R$. The red edges can be decomposed into $(3t-n)$ vertex-disjoint graphs $R_\ell$, $\ell\in[3t-n]$, where $R_\ell$ is a copy of $\Gamma_k$ if $\ell\leq (k-1)n-(3k-4)t$, and a copy of $\Gamma_k\setminus\set{1,k,2k}$ otherwise. It is not hard to see that every independent set of $\Gamma_k$ of size~$k$ contains a vertex from $\set{1,k,2k}$. Therefore $\alpha(R_\ell)\leq k-1$ if $\ell> (k-1)n-(3k-4)t$, and $\alpha(R_\ell)\leq k$ otherwise. It follows that 
\[\alpha(R)\leq \bigcup_{\ell\in[3t-n]} \alpha(R_\ell)\leq (3t-n)\cdot(k-1)+[(k-1)n-(3k-4)t]\cdot 1 = t,\]
and therefore 
\[g(n;3,cn+1) = g(n;3,t+1)\geq |P| = e(\Gamma(n;k,t))-|R| \geq  e(\Gamma(n;k,t))-3k^2n.\]
We obtain the desired result as~$k$ depends only on $c$.
\end{proof}

We are now ready to prove~\cref{thm:TriangleCase}\ref{item:TriangleLinearCase}, that is $g(n;3,cn) = (1-o(1)) \RT{3}{cn}{n}$ when $c\in \left(0,1\right]\setminus \left(\frac{1}{3},\frac{4}{11}\right)$.

\begin{proof}[Proof of~\cref{thm:TriangleCase}\ref{item:TriangleLinearCase}.]
Recall that, for $c\ge 1/2$, the theorem statement is a simple consequence of 
Tur\'{a}n's theorem and the blow-up colouring, see~\cref{thm:LinearLargeC}.
When $c\leq1/3$, \cref{thm:TriangleCase}\ref{item:TriangleLinearCase} follows immediately by combining~\cref{thm:TriangleLinearSmall,eq:upper_bound_RT,thm:RTBrandt}. For the remaining range, let $k\geq 2$ be the unique integer such that $k/(3k-1)\leq c < (k-1)/(3k-4)$. Let~$n$ be large enough, in particular, such that $t:=\ceil{cn}$ satisfies $kn/(3k-1)\leq t<t+1 < (k-1)n/(3k-4)$. We obtain from~\cref{thm:TriangleIntermdC} that, if~$k$ is such that~\cref{conj:Andrasfai} holds we have
\begin{align*}
    e(\Gamma(n;k,t-1))-O(n)\leq g(n;3,cn) \leq \RT{3}{cn}{n} = e(\Gamma(n;k,t-1)).
\end{align*}
If $c\in(4/11,1/2)$, we obtain $k\in\{2,3,4\}$, hence~\cref{conj:Andrasfai} holds following the work of Andr{\'a}sfai~\cite{andrasfai1962extremalproblem} and {\L}uczak, Polcyn, and Reiher~\cite{luczak2022ramsey,luczak2025next}. \Cref{thm:TriangleCase}\ref{item:TriangleLinearCase} now follows since $e(\Gamma(n;k,t-1))=\Theta(n^2)$, see~\cref{eq:SizeCanonicalBlowUp}.  
\end{proof}

\section{The triangle case with sublinear independence number}\label{sec:TriangleSubLinear}

We first prove~\cref{obs:ADMR}, showing that $g(n;3,t)$ is bounded away from $\RT{3}{t}{n}$ when $t$ is a constant factor away from $\sqrt{n\log n}$. The main ingredient of this observation is the fact that every triangle-free graph of average degree $d$ contains an independent set of size at least $n (\log d -1)/d$, which was proved by Shearer~\cite{shearer1983} to obtain the upper bound in~\eqref{eq:bounds-R3t}.

\begin{proof}[Proof of~\cref{obs:ADMR}]
    Let~$n$ be large enough and assume that $g(n;3,t) \geq (1-\delta)\cdot\RT{3}{t}{n}$. Let \RBP{} be a $(3,\ceil{t})$-free colouring of $K_n$ such that $|P|=g(3;n,t)$. Given that $\alpha(R\cup P)\leq \alpha(R)< t$, we have $|R\cup P|\leq \RT{3}{t}{n}$ and it follows that $|R|\leq \delta\cdot \RT{3}{t}{n}$. By~\cref{eq:trivial_bound}, we obtain \(|R| \leq \delta nt/2,\) hence $R$ is a triangle-free graph with average degree at most $\delta t$. Shearer~\cite[Theorem~1]{shearer1983} showed that $R$ then contains an independent set of size at least $n (\log(\delta t) -1)/(\delta t)$. This leads to the desired contradiction since for $t=\gamma\sqrt{n\log n}$ and~$n$ large enough we have
    \[n\; \frac{\log(\delta t) -1}{\delta t}
    =\frac{\sqrt{n}}{\delta\gamma\sqrt{\log n}}
    \Big(\frac{1}{2}\,\log n + \frac{1}{2}\,\log \log n+ \log(\delta\gamma) -1\Big)
    \geq\frac{1}{2\delta\gamma}\sqrt{n\log n}\geq t.\qedhere\]
\end{proof}

The triangle-free process is a key tool in our analysis of $g(n;3,t)$ when $t$ is sublinear. This random graph process is defined iteratively, starting with an empty graph $G_0$ on~$k$ vertices, and adding edges, one by one, chosen uniformly at random from all those that do not close a triangle. Denote by $G_{i}$ the random graph obtained after adding $i$ edges, and by $G_{k,\triangle}$ the final random graph resulting from this process. 

The triangle-free process was introduced by Bollob{\'a}s and Erd\H{o}s (see e.g.~\cite{bollobas2009random}), and Spencer~\cite{spencer1995maximal} conjectured that it should establish $R(3,t) = \Omega(t^2/ \log t)$ by producing a dense-enough triangle-free graph with small independence number. This was proved by Bohman~\cite{bohman2009triangle} in 2009. Four years later, Fiz Pontiveros, Griffiths, and Morris~\cite{fpgm2013}, and independently Bohman and Keevash~\cite{bohman2013dynamic}, used a much more detailed analysis of the triangle-free process to prove \(R(3,t) \geq \parens*{1/4-o(1)}t^2/\log t.\) Both refined analyses of the triangle-free process~\cite{fpgm2013,bohman2013dynamic} established the following.

\begin{theorem}[Bohman and Keevash~\cite{bohman2013dynamic},
Fiz Pontiveros, Griffiths, and Morris~\cite{fpgm2013}]\label{thm:OutputTriangleFreeProcess}
    For every $\varepsilon>0$, there exists $k_0$ such that for all $k\geq k_0$ there exists a triangle-free graph $G=G(k,\varepsilon)$ on~$k$ vertices satisfying  
\[\alpha(G)\leq (1+\varepsilon)\sqrt{2k\log k}, \quad\text{ and } \quad e(G)\geq \frac{1}{2\sqrt{2}}(1-\varepsilon)k^{3/2}\sqrt{\log k}.\]
\end{theorem}

We split the proof of~\cref{thm:TriangleCase}\ref{item:TriangleSublinearCase} into two subcases. First, when $t$ is sufficiently larger than $4\sqrt{n\log n}$, we use the graph produced by~\cref{thm:OutputTriangleFreeProcess} as a black box. Then, for smaller $t$ we slightly modify the triangle-free process to produce a desired colouring. For both subcases, we prove a stronger statement, relating $g(n;3,t)$ to $nt/2$ instead of $\RT{3}{t}{n}$. The results follow immediately by~\cref{eq:trivial_bound}. 

\begin{theorem}\label{thm:sublinear_triangle}
For every $\gamma>4$ there exist $n_0\ge 1$ and $\delta\in(0,1/2)$ such that for all $n\geq n_0$ and $t=\gamma \sqrt{n\log n}$ we have 
\[g(n;3,t)\geq \delta \, \frac{nt}{2}.\]
Furthermore, if $t\gg\sqrt{n\log n}$ then 
\[g(n;3,t)\geq \Big(\frac12-o(1)\Big) \, \frac{nt}{2}.\]
\end{theorem}

\begin{proof}
We prove the first part of the theorem by showing that for any constant $\gamma>4$, there exist $\delta<1/2$, $\varepsilon\in(0,1)$, $\zeta>2$ and $n_0\geq 1$ such that the following holds for any $n\geq n_0$. There exists an integer $k\approx n/\zeta^2$ such that the $n$-blow-up colouring of the triangle-free-process $G=G(k,\varepsilon)$ as guaranteed by~\cref{thm:OutputTriangleFreeProcess} yields a colouring \RBP{} such that 
\[\alpha(R)<\gamma\sqrt{n\log n}, \quad \text{ and }\quad 
|P|\geq \delta \cdot \frac{\gamma\; n^{3/2}\sqrt{\log n}}{2}.\]
Let $\gamma>4$, and let $\varepsilon>0$ be sufficiently small such that
\[\varepsilon<\frac14,\text{ and }\qquad \gamma>\frac{4(1+\varepsilon)}{1-\varepsilon}.\]
Let $\zeta>2$ be such that 
\begin{equation}\label{eq:defzeta}
\gamma=\sqrt{2}\; \frac{1+\varepsilon}{1-\varepsilon}\Big(\zeta+\frac{1}{\zeta}\Big).
\end{equation}
We note that $\gamma>4(1+\varepsilon)/(1-\varepsilon)$ implies that such a choice of $\zeta$ is always possible, that $\zeta$ is increasing with $\gamma$, and that we indeed have $\zeta>2$. 

Let~$n$ be large enough, $t=\gamma\sqrt{n\log n}$, and let~$k$ be the smallest integer such that $n\leq \zeta^2 k$. Then  \(\zeta^2(k-1)<n\leq \zeta^2k\), and hence for large enough~$k$, 
\begin{equation}\label{eq:boundzeta}
    \zeta^2\sqrt{1-\varepsilon}<\zeta^2\Big(1-\frac{1}{k}\Big)<\frac{n}{k}\leq \zeta^2.
\end{equation}
Let $G=G(k,\varepsilon)$ be a graph as guaranteed by~\cref{thm:OutputTriangleFreeProcess}, that is a triangle-free graph on~$k$ vertices such that
\begin{align}\label{eq:PropsTFP}
\alpha(G)\leq (1+\varepsilon)\sqrt{2k\log k},\quad \text{ and }\quad e(G)\geq \frac{1}{2\sqrt{2}}(1-\varepsilon)k^{3/2}\sqrt{\log k}.
\end{align}
Let \RBP{} be an $n$-blow-up colouring of $G$. By~\cref{thm:PropPurpleBlowUp}, $R\cup P$ is a triangle-free graph such that 
\begin{align*}
    \alpha(R)\leq \alpha(G) \ceil*{\frac{n}{k}}
    &\leq (1+\varepsilon)\sqrt{2k\log k}\Big(\frac{n}{k}+1\Big) \\
    &\leq (1+\varepsilon)\sqrt{2n\log n}\Big(\zeta+\frac{1}{\zeta(1-\varepsilon)^{1/4}}\Big)\\ 
    &\leq \frac{1+\varepsilon}{1-\varepsilon}\sqrt{2n\log n}\Big(\zeta+\frac{1}{\zeta}\Big)    \\
    &=\gamma\sqrt{n\log n}, 
\end{align*}
where in the third inequality we use lower and upper bounds of~\cref{eq:boundzeta}, and the last equality is simply the definition of $\zeta$. 
In order to bound the number of purple edges, observe that for~$n$ large enough, $k\ge n/\zeta^2$ and $\zeta=O(1)$ imply that $\log k\ge \log n /2$, say. \cref{thm:PropPurpleBlowUp} and \cref{eq:PropsTFP} then imply that 
\begin{align*}
    |P|&\geq e(G)\Big(\frac{n^2}{k^2}-\frac{3n}{k}\Big)
    \ge \frac{1}{8}k^{3/2} \sqrt{\log k}\Big(\frac{n^2}{k^2} -\frac{3n}{k}\Big)\\
    &\ge \frac{1}{2^4\zeta^3}n^{3/2} \sqrt{\log n} \Big(\frac{n^2}{k^2} -\frac{3n}{k}\Big)\\
    &\ge \frac{1}{2^3\zeta^3\gamma} \Big(\frac{n^2}{k^2} -\frac{3n}{k}\Big) \frac{tn}{2}.
\end{align*}
Now, $n/k\ge \zeta^2\sqrt{1-\eps}>3$, by \cref{eq:boundzeta} and since $\eps<1/4$ and $\zeta>2$. It follows that $|P|\ge \delta tn/2$ for some $\delta=\delta (\gamma,\zeta)>0.$ 

For the second part, when $t\gg\sqrt{n\log n}$, we reason similarly, with some adjustments since we need more precise estimates on $\delta$. Observe that we can assume that $t=o(n)$ as otherwise the result follows from \cref{thm:TriangleCase}\ref{item:TriangleLinearCase}. 

Let $\varepsilon>0$, let~$n$ be large enough and let $\gamma=t/\sqrt{n\log n}$. Recall that here we assume $\gamma\to\infty$ as $n\to \infty$. Define $\zeta$ by \[\gamma=\sqrt{2}\; \frac{(1+\varepsilon)^2}{1-\varepsilon}\Big(\zeta+\frac{1}{\zeta}\Big)\Big(1-\frac{2\log\zeta}{\log n}\Big)^{1/2}.\] 
Note that $\zeta=\Theta (\gamma)$, 
and thus our assumptions on $t$ imply that $1\ll \zeta\ll \sqrt{n/\log n}.$
Again, let~$k$ be the smallest integer such that $n\leq \zeta^2 k$. 
Then $k\leq (n/\zeta^2)^{(1+\varepsilon)^2}$  for~$n$ large enough, and thus, 
\begin{equation*}
    \sqrt{\log k}\leq (1+\varepsilon)\sqrt{\log n}\cdot\Big(1-\frac{2\log\zeta}{\log n}\Big)^{1/2}.
\end{equation*}
The computations of $\alpha(R)$ and $|P|$ are now similar to the first case. From \cref{thm:PropPurpleBlowUp,eq:PropsTFP} we deduce that  
\begin{align*}
    \alpha(R)
    \leq (1+\varepsilon)\sqrt{2k\log k}\ceil*{\frac{n}{k}}
    \leq \frac{(1+\varepsilon)^2}{1-\varepsilon}\sqrt{2n\log n}\Big(\zeta+\frac{1}{\zeta}\Big)\Big(1-\frac{2\log\zeta}{\log n}\Big)^{1/2}   
    =\gamma\sqrt{n\log n},
\end{align*}
and 
\begin{align}\label{aux523}
|P| &\geq e(G)\Big(\frac{n^2}{k^2}-\frac{3n}{k}\Big)
\geq \frac{1-\varepsilon}{2\sqrt{2}}\, k^{3/2}\sqrt{\log k}\, \frac{n^2}{k^2}(1-o(1))\nonumber \\
&= \frac{1-\varepsilon}{\gamma \sqrt{2}}\, \Big(\frac{n \log k}{k\log n}\Big)^{1/2}\, \frac{\gamma n^{3/2}\sqrt{\log n}}{2}(1-o(1)),
\end{align}
where we used that $k=o(n)$ in this case. Now $n/k\ge \zeta^2(1-1/k)=\zeta^2(1-o(1))$ and
\begin{equation*}
    \sqrt{\log k}\geq \sqrt{\log n}\cdot\Big(1-\frac{2\log\zeta}{\log n}\Big)^{1/2}, 
\end{equation*}
by definition of $k$. Thus, \cref{aux523} implies that 
\begin{align} |P| &\ge (1-o(1)) \frac{(1-\varepsilon)\zeta}{\gamma \sqrt{2}}\, \Big(1-\frac{2\log\zeta}{\log n}\Big)^{1/2}\, 
\frac{tn}{2} \\
&= (1-o(1)) \frac{(1-\eps)^2}{(1+\eps)^2}\, \frac{tn}{4},
\end{align}
by definition of $\zeta$ and since $\zeta\to\infty$. The theorem follows since $\eps>0$ is arbitrary.
\end{proof}


To cover the remaining range of $t$, that is $t=\gamma\sqrt{n\log n}$ for  $\sqrt{2}<\gamma \le 4$, we need a much more refined analysis of the triangle-free process. Recall that we denote by $G_i$ the random graph obtained from the triangle-free process after adding $i$ edges. Furthermore, for a given $\varepsilon>0$ we let 
\begin{equation}\label{eq:defmstar}
m^*(\varepsilon) := \Big(\frac{1}{2\sqrt{2}}-\varepsilon\Big)n^{3/2}\sqrt{\log n}.    
\end{equation}

Fiz Pontiveros, Griffiths, and Morris~\cite{fpgm2013} defined a sequence of {\em good} events, denoted here by~$\mathcal{E}(m)$ for simplicity, such that if~$\mathcal{E}(m)$ holds then there are still free edges that do not close a triangle with $G_m$, implying that the triangle-free process has not stopped yet. They proved the following proposition. 
\begin{theorem}[Fiz Pontiveros, Griffiths, and Morris~\cite{fpgm2013}~Theorems~6.9~and~7.2]\label{thm:FPGM}
    For every sufficiently small constant $\varepsilon>0$, and every~$n$ large enough, let $m=m^*(\varepsilon)$. Then,
    \begin{enumerate}
        \item $\mathcal{E}(m)$ holds with probability at least $1-n^{-\log n}$.\label{thm:FPGM_E}
        \item With probability at least $1-e^{-\sqrt{n}}$ either $\mathcal{E}(m)$ fails or
    \(\alpha(G_{m})< \parens*{\sqrt{2}+10\sqrt{\varepsilon}}\sqrt{n\log n}\).\label{thm:FPGM_alpha}
    \end{enumerate}
\end{theorem}
We remark that~\cite[Proposition 7.2]{fpgm2013} is stated in terms of $G_{n,\triangle}$, that is the final random graph produced by the triangle-free process. However, the stronger statement of~\cref{thm:FPGM}\ref{thm:FPGM_alpha} is easily extracted from the proof in~\cite[Section 7.7]{fpgm2013}.
We use these results to prove the following theorem which, together with~\cref{eq:trivial_bound}, concludes the proof of~\cref{thm:TriangleCase}\ref{item:TriangleSublinearCase}.

\begin{theorem}\label{thm:TriangleAtThresholdPrecise}
For every $\varepsilon>0$, there exists a constant $\delta>0$ such that for~$n$ large enough and $t=(\sqrt{2}+\varepsilon)\sqrt{n\log n}$ we have
    \(g(n;3,t)\geq \delta nt/2.\)
\end{theorem}

\begin{proof} 
    Observe that, by monotonicity, it is sufficient to prove the statement for $\varepsilon$ small enough. Therefore let $\varepsilon>0$ be small enough and $n$ be large enough such that~\cref{thm:FPGM} holds. Let $\varepsilon_1:=\varepsilon^2/100$, and fix a real number $\varepsilon_2\in(0,\varepsilon_1)$. Let $m_1=m^*(\varepsilon_1)$ and $m_2=m^*(\varepsilon_2)$ as given by~\cref{eq:defmstar}. Our goal is to follow the triangle-free process up to $m_2$, colouring all $m_1$ first edges in red, and then the remaining $m_2-m_1$ edges in purple. 

    We first condition the triangle-free process on $\mathcal{E}(m_2)$ holding. Because we have $m_2>m_1$, we know that $\mathcal{E}(m_1)$ holds, and therefore by~\cref{thm:FPGM}, 
    \[\mathbb{P}[\alpha(G_{m_1})\geq t\, |\,\mathcal{E}(m_2) ]\leq\frac{e^{-\sqrt{n}}}{1-n^{-\log n}}=o(1).\]
    We therefore select an outcome of the triangle-free process such that $\mathcal{E}(m_2)$ holds and $\alpha(G_{m_1})<t$. We colour all edges in $G_{m_1}$ in red, and all edges in $G_{m_2}\setminus G_{m_1}$ in purple. We obtain that $R\cup P$ is a triangle-free graph such that \(\alpha(R)=\alpha(G_{m_1})<  t\) and \(|P|=(\varepsilon_1-\varepsilon_2) n^{3/2}\sqrt{\log n}\). Therefore, 
    \[g(n;3,t)\geq (\varepsilon_1-\varepsilon_2) n^{3/2}\sqrt{\log n} = \delta nt/ 2,\]
    with \(\delta:=2(\varepsilon_1-\varepsilon_2)/(\sqrt{2}+\varepsilon)\).    
\end{proof}

\section{General clique sizes}\label{sec:GeneralS}

\subsection{Linear independence number}

This section is dedicated to the proof of~\cref{thm:GeneralSmallLinear}, showing that $g(n;s;t)$ is asymptotically equal to the Ramsey-Tur\'an number when $s$ is fixed, $t=cn$ and $c$ is small enough with respect to $1/s$. We will prove this statement in two steps, covering the case when $s$ is odd first, and then when $s$ is even.~\cref{thm:GeneralSmallLinear} follows easily using the bounds on $\RT{s}{cn}{n}$ from~\cref{thm:LudersReiher}.

\begin{theorem}\label{thm:LinearOdd}
    For any integer $\ell\geq 1$, and any constant $0<c<\frac{1}{3\ell}$, we have
    \[g(n;2\ell+1,cn) \geq\parens*{\frac{\ell-1}{\ell}+c+o(1)} \frac{n^2}{2}.\]
\end{theorem}

\begin{proof}

     Fix $\eta>0$. Let $d,k$ be non-negative integers such that $(c-\eta)\ell< \frac{d}{k}< c\ell$, and such that there exists a triangle-free $d$-regular graph $G$ on~$k$ vertices with independence number $d$, as guaranteed by~\cref{thm:RTBrandt}. For~$n$ large enough, assume without loss of generality that $\ell$ divides~$n$ and let $\RBP$ be an $(n/\ell)$-blow-up colouring of $G$. It follows from~\cref{thm:PropPurpleBlowUp} that $R\cup P$ is a triangle-free graph and, since~$n$ is large enough, $R$ has independence number at most
     \[
        \ceil*{\frac{n}{\ell k}}\alpha(G)
        \leq \Big(\frac{d}{\ell k}+\frac{d}{\ell n}\Big)n
        < cn.
     \]
     We additionally obtain from~\cref{thm:PropPurpleBlowUp} that for~$n$ large enough, 
     \begin{align*}
         \size{P}\geq e(G)\Big(\frac{n^2}{\ell^2 k^2}-\frac{3n}{\ell k}\Big) = \frac{n^2}{2\ell} \Big(\frac{d}{\ell k}-\frac{3d}{2n\ell}\Big) > \frac{n^2}{2\ell}(c-\eta),
     \end{align*}
    and
    \[|R|\leq e(G)\frac{n/\ell+k}{k} = O(n).\] 
    
    We define a colouring \RBPprime{} of $F\cong K_n$ as follows. We partition the vertices of $F$ into $\ell$ disjoint sets $V_1,\ldots,V_{\ell}$ of size $n/\ell$. For each $i\in[\ell]$, we embed a copy of the colouring \RBP{} on the vertices $V_i$. For each $i\neq j\in[\ell]$ we colour the edges between $V_i,V_j$ in red and purple (with no blue edges) such that the resulting $R'$ will be isomorphic to the so-called strong product $R\boxtimes K_\ell$. Specifically, for $i\neq j$, $v_i\in V_i$ and $v_j\in V_j$, colour the edge $v_iv_j$
    \begin{itemize}
        \item red if $v_i$ and $v_j$ are copies of the same vertex of the colouring \RBP{};
        \item red if $v_i$ and $v_j$ are copies of distinct vertices $u,w$ in \RBP{} such that the edge $uw$ is in $R$; and
        \item purple otherwise.
    \end{itemize}

    \medskip 
    
    First note that $R'\cup P'$ is $K_{2\ell+1}$-free, since $(R'\cup P')[V_i]$ is a triangle-free graph for every $i\in\ell$.  
    Next, it is not hard to see that $\alpha(R') \le \alpha(R)$. Indeed, every $U\se V(F)$ that is independent in $R'$ can contain at most one copy of $u$, for each $u\in V(\RBP)$. Thus, $U=\{u_1^{j_1},\ldots,u_h^{j_h}\}$, where $u_1,\ldots, u_h$ are distinct vertices of $V(\RBP)$, and $u_i^{j_i}$ is the copy of $u_i$ in $V_{j_i}$ for every $1\le i \le h$. But this set $U$ is independent in $R'$ if and only if the set of vertices $\{u_1,\ldots,u_h\}$ is an independent set in $R$. Therefore, $\alpha(R') \le \alpha(R) < cn$.  
    To bound the number of purple edges from below, note that for any two distinct $i, j\in[\ell]$, the red edges between $V_i$ and $V_j$ can be decomposed into a perfect matching, with the addition of the edges $u_iw_j$ and $u_jw_i$ for every edge $uw$ of $R$ with copies $u_iw_i$ in $V_i$ and $u_jw_j$ in $V_j$. Therefore, there are exactly $2|R|+n/\ell$ red edges between $V_i$ and $V_j$, for all distinct $i, j\in[\ell]$. It follows that 
    \begin{align*}
    g(n;2\ell+1,cn)\geq |P'| &= \binom{\ell}{2}\brackets*{\parens*{\frac{n}{\ell}}^2 - 2|R|-\frac{n}{\ell}} + \ell|P|\\    
    &\geq \frac{n^2}{2}\Big(\frac{\ell-1}{\ell} + c - \eta\Big) -O(n).
    \end{align*}
Since $\eta>0$ can be chosen arbitrarily small, we obtain the desired result.    
\end{proof}


We now turn our attention to the case of even cliques. Modifying the Bollob{\'a}s--Erd\H{o}s construction~\cite{bollobas1976ramsey} for the study of $\RT{4}{t}{n}$, Fox, Loh, and Zhao~\cite{fox2015critical} proved the existence of the following graph.

\begin{theorem}[Fox, Loh, and Zhao~\cite{fox2015critical}]\label{thm:ExistenceGraphFox}
    For every $0<a<1/2$, for every even integer~$k$, large enough such that $a \geq \frac{1}{\varepsilon k}+\varepsilon$, where $\varepsilon = \frac{4(\log\log k)^{3/2}}{\log^{1/2} k}$, there exists a graph $G=G(a)$ on~$k$ vertices such that following properties hold.
    \begin{enumerate}
        \item $G$ is $K_4$-free,
        \item $\alpha(G)<ak$,
        \item $G$ has at least $\frac{1}{2}(\frac{1}{4}+a-a^2-2\varepsilon)k^2$ edges.
    \end{enumerate}
\end{theorem}

The quantitative relation between~$\varepsilon$,~$a$ and~$k$ is not explicit in~\cite[Theorem 1.7]{fox2015critical} but is easily extracted from their proof. We use the graphs from~\cref{thm:ExistenceGraphFox} as the starting point for our construction when $s$ is even, but using blow-up colourings of these graphs is far from sufficient to prove~\cref{thm:LinearEven}, as we will need to ``sprinkle'' sparse random graphs to obtain asymptotically tight bounds.

\begin{theorem}\label{thm:LinearEven}
    For any integer $\ell\geq 2$, and any constant $c<1/(3\ell-2)$, we have
    \[g(n;2\ell,cn)\geq\frac{1}{2}\parens*{\frac{3\ell-5}{3\ell-2}+c-c^2-o(1)}n^2.\]
\end{theorem}

\begin{proof}
We take a blow-up colouring of a $K_4$-free graph as guaranteed by~\cref{thm:ExistenceGraphFox}, together with $(\ell-2)$ disjoint copies of a blow-up colouring of a $K_3$-free graph as guaranteed by~\cref{thm:RTBrandt}, with almost only purple edges in between these graphs except for a small random set of red edges, to form a $K_{2\ell}$-free graph with the desired properties.

Let $\eta>\delta>0$ be arbitrarily small constants. Let~$n$ be large enough, and without loss of generality, divisible by $3\ell-2$. Let 
\[ k=2\floor{\log \log n},\quad  \varepsilon=\frac{4(\log\log k)^{3/2}}{\log^{1/2}k}, \quad a=\frac{(3\ell-2)}{4}(c-\eta).
\]
For~$n$ large enough, observe that $a \geq \frac{1}{\varepsilon k}+\varepsilon$. Therefore, there exists a $K_4$-free graph~$G_4$ be on~$k$ vertices such that 
\[e(G_4)\geq \frac{1}{2}\Big(\frac{1}{4}+a-a^2-2\varepsilon\Big)k^2\quad \text{and}\quad \alpha(G_4)\leq ak,\]
by~\cref{thm:ExistenceGraphFox}. Let $\RBP[4]$ be a $(4n)/(3\ell-2)$-blow-up colouring of~$G_4$. It follows from~\cref{thm:PropPurpleBlowUp} that $R_4\cup P_4$ is a $K_4$-free graph such that
\begin{align}
  \alpha(R_4)\leq \ceil*{\frac{4n}{k(3\ell-2)}}\alpha(G_4)\leq \frac{4na}{(3\ell-2)}+ak
=(c-\eta)n+ak<(c-\delta)n,\label{eq:alphaR4}
\end{align}
for~$n$ large enough. We also obtain from~\cref{thm:PropPurpleBlowUp,thm:ExistenceGraphFox} that the number of purple edges is given by
\begin{align}
|P_4|&\geq e(G_4)\Big(\frac{(4n)^2}{k^2(3\ell-2)^2}-\frac{12n}{k(3\ell-2)}\Big)
\geq\frac{1}{2}\Big(\frac{1}{4}+a-a^2-o(1)\Big)\frac{(4n)^2}{(3\ell-2)^2}\nonumber\\
\label{eq:P4}
&=\frac{n^2}{2}\Big(\frac{4}{(3\ell-2)^2}+\frac{4(c-\eta)}{(3\ell-2)}-(c-\eta)^2-o(1)\Big).    
\end{align}
Observe that if $\ell=2$, we have $3\ell-2=4$ and we are done since $\eta$ is arbitrarily small. Assume now that $\ell>2$, and let $d$, $h$ be integers such that 
\[\frac{(3\ell-2)}{3}(c-\eta)<\frac{d}{h}<\frac{(3\ell-2)}{3}(c-\delta)<\frac13\] 
and such that~\cref{thm:RTBrandt} guarantees the existence of a triangle-free $d$-regular graph $G_3$ on $h$ vertices with independence number $d$. Let $\RBP[3]$ be a $(3n)/(3\ell-2)$-blow-up colouring of $G_3$. It follows from~\cref{thm:PropPurpleBlowUp} that for~$n$ large enough, $R_3\cup P_3$ is a triangle-free graph with
\begin{equation}\label{eq:alphaR3}
    \alpha(R_3)\leq \ceil*{\frac{3n}{h(3\ell-2)}}\alpha(G_3)\leq\Big(\frac{3d}{h(3\ell-2)}+\frac{d}{n}\Big)n <(c-\delta)n,    
\end{equation}
and 
\begin{equation}\label{eq:P3}
|P_3|\geq e(G_3)\Big(\frac{(3n)^2}{h^2(3\ell-2)^2}-\frac{9n}{h(3\ell-2)}\Big)\geq\frac{3(c-\eta)}{2(3\ell-2)}n^2.
\end{equation}

We now define a colouring \RBP{} of $F\cong K_n$ as follows. Informally, we take a copy of $\RBP[4]$ and $(\ell-2)$ vertex-disjoint copies of $\RBP[3]$. The edges between any two of these copies are coloured red randomly, with probability $p$, and purple with probability $1-p$. We show that the independence number of $R$ is not much larger than the one of each~$R_i$. 

More formally, partition the vertices of $F$ into one set $V_1$ of size $(4n)/(3\ell-2)$ and $\ell-2$ sets $V_2,\ldots,V_{\ell-1}$ of size $(3n)/(3\ell-2)$ each. Embed a copy of the colouring $\RBP[4]$ into $F[V_1]$, and for each $i\in\{2,\ldots,\ell-1\}$, embed a copy of $\RBP[3]$ in $F[V_i]$. Let $p\in (0,1)$ be such that $1/n\ll p \ll 1$, and for $i,j \in \binom{[\ell-1]}{2}$, colour each edge between $V_i$ and $V_j$ in red independently at random with probability $p$, and purple otherwise.

\begin{claim}\label{clm:prob}
    For every $\mu>0$, the following properties are satisfied with probability $1-o(1)$ as $n\to\infty$.
    \begin{enumerate}[label=(P\arabic*)]
        \item There are $o(n^2)$ red edges between any two distinct sets $V_i$ and $V_j$.\label{item:nbedges}
        \item For all sets $X\subseteq V_i$ and $Y\subseteq V_j$, with $i\neq j$, such that $|X||Y|\geq \mu n^2$, there exists a red edge between $X$ and $Y$.\label{item:crossedges}
    \end{enumerate}
\end{claim}

\begin{proof}
    \cref{item:nbedges} is true for $p\ll 1$ by standard concentration results on the number of edges in the binomial random graph. Fix two sets $X\subseteq V_i$ and $Y\subseteq V_j$ such that $|X||Y|\geq \mu n^2$. The probability that there is no red edge between $X$ and $Y$ is \((1-p)^{|X||Y|}\leq \exp\parens*{-\mu p n^2}\). There are at most $2^{2n}$ choices of sets $X,Y$.~\cref{item:crossedges} follows immediately from the union bound.
\end{proof}

Fix a colouring satisfying~\cref{item:nbedges,item:crossedges}. Observe that any set of $2\ell$ vertices in $F$ contains either $3$ vertices from $V_i$ for some $i\in\set{2,\ldots,\ell-1}$ or $4$ vertices from $V_1$. Given that $(R\cup P)[V_i]$ is $K_{4}$-free if $i=1$ and $K_3$-free if $i\in\set{2,\ldots,\ell-1}$, we obtain that $R\cup P$ is $K_{2\ell}$-free. Assume now that $\alpha(R)\geq cn$, and so there exists a set of vertices $I$ with $|I|\geq cn$ and no red edges in $F[I]$. By the pigeonhole principle, there exists $i\in[\ell-1]$ such that $X=I\cap V_i$ contains at least $cn/(\ell-1)$ vertices. Fix such an~$i$. 
Now $\alpha(R[V_i]) < (c-\delta)n$, by~\cref{eq:alphaR4,eq:alphaR3}. Therefore,~$I$ contains at least $\delta n$ vertices not in~$V_i$. Once again, by the pigeonhole principle there exists $j\in[\ell-1]\setminus\{i\}$ such that $Y=I\cap V_j$ has size at least $\delta n/(\ell-2)$. By~\cref{item:crossedges} there exists a red edge between the two sets $X,Y$, contradicting the fact that $F[I]$ contains no red edges.

Finally, to bound the total number of purple edges in $P$ from below, observe that~\cref{item:nbedges} implies that the number of purple edges between $V_i$ and $V_j$, for $i, j\in[\ell-1]$ with $i\neq j$, is at least $|V_i||V_j|-o(n^2)$. Given that any purple edge in $P$ is either between $V_i$ and $V_j$ for some distinct $i,j\in[\ell-1]$, in $P_4$, or in one of the $\ell-2$ copies of $P_3$, we obtain from~\cref{eq:P4,eq:P3} that 
\begin{align*}
    |P| &\geq |P_4| + (\ell-2)|P_3|
    +\binom{\ell-2}{2}\frac{9n^2}{(3\ell-2)^2}
    +(\ell-2)\frac{12n^2}{(3\ell-2)^2}-o(n^2)\\
    &\geq\frac{1}{2}\Big(\frac{3\ell-5}{3\ell-2}+(c-\eta)-(c-\eta)^2-o(1)\Big)n^2,
\end{align*}
and the result follows as $\eta$ is arbitrarily small. 
\end{proof}

We note that, in~\cref{clm:prob}, we could use the better estimate $\Theta(pn^2)$ for red edges in the random colourings between $V_i$'s and $V_j$'s, and hence increase the second order term in $|P|$. However, the error term $\varepsilon$ from~\cref{thm:ExistenceGraphFox} used for $P_4$ leads to an error term in $|P|$ that is at least of order $\Theta\parens*{n^2/\sqrt{\log n}}$.

\subsection{A probabilistic bound}

In this section, we prove~\cref{thm:ProbabilisticLowerBound} and show that for any $s$, the behaviour of $g(n;s,t)$ closely follows $\RT{s}{t}{n}$ up to some logarithmic offset in the function $t$. 

\begin{proof}[Proof of~\cref{thm:ProbabilisticLowerBound}]
For some absolute constant $0<p<1$ to be determined later, let $s\geq 3$ be a fixed integer and let~$n$ be large enough. Let $t\in\brackets*{R^{-1}(s,n)\log n,n}$ and let $G$ be a $K_s$-free graph on~$n$ vertices, with $\alpha(G)< t/\log n$ and $e(G) = \RT{s}{t/\log n}{n}$. Let $G_p$ be the spanning subgraph of~$G$ obtained by randomly selecting the edges of $G$ with probability $p$, independently of each other. Let \RBP{} be the colouring of $K_n$ defined by
\[\begin{cases}
    P &= E(G_p),\\
    R &= E(G)\setminus E(G_p),\\
    B &= E(K_n)\setminus E(G).\\
\end{cases}\]
By standard concentration results, we know that with high probability 
\begin{equation}\label{eq:ConcentrationP}
|P|\geq \mathbb{E}\big[|P|\big]/2 = p/2\cdot\RT{s}{t/\log n}{n}.    
\end{equation}
Observe that $R\cup P$ is isomorphic to $G$ and therefore a $K_s$-free graph. By the union bound,
\begin{align*}
\Pr[\alpha(R)\geq t] &
\leq \sum_{\substack{T\subseteq V(G)\\|T|=t}}\Pr\brackets*{T\text{ is independent in }R} 
= \sum_{\substack{T\subseteq V(G)\\|T|=t}} p^{e(G[T])}.
\end{align*}
For any such fixed set $T$, let $F=G[T]$. Then $F$ is a graph on $t$ vertices with $\alpha(F)\leq \alpha(G)< t/\log n$. Thus by \cref{thm:Turan} applied to the complement of $F$, 
\[e(F)\geq \alpha(F)\frac{(t/\alpha(F))\cdot(t/\alpha(F)-1)}{2}\geq \frac{t^2}{4\alpha(F)}\geq \frac{t \log n}{4}.\]
Therefore for $n$ large enough, by the union bound,
\begin{align*}
  \Pr[\alpha(R)\geq t]&\leq \binom{n}{t}p^{(t \log n)/4} \leq \exp\brackets*{\frac{t\log n}{4}\Big(4+\log p\Big)}.
\end{align*}
Thus for $p<e^{-4}$, with high probability $R$ contains no independent set of size $t$. Since \cref{eq:ConcentrationP} also holds with probability $1-o(1)$, 
there exists an $(s,t)$-free colouring \RBP{} of $K_n$ such that $|P|\geq p/2\cdot\RT{s}{t/\log n}{n}$, 
and the claim follows.
\end{proof}

Assuming~\cref{conj:OnRamsey}, Balogh, Hu, and Simonovits~\cite{balogh2015phase} showed the following.
\begin{theorem}[Balogh, Hu, and Simonovits~\cite{balogh2015phase}
]\label{theorem:BHSphase}
If~\cref{conj:OnRamsey} is true for some $2\leq \ell\leq s$, then
    \[\RT{s}{R^{-1}(\ell,n)}{n}=\frac{1}{2}\Big(1-\frac{1}{
    \floor*{
    \frac{s-1}{\ell-1}
    }
    }\Big)n^2+o(n^2).\]
\end{theorem}

We now prove~\cref{prop:Generalbehaviour_precise} by combining~\cref{theorem:BHSphase} with the lower bound from~\cref{thm:ProbabilisticLowerBound}. 

\begin{proof}[Proof of~\cref{prop:Generalbehaviour_precise}]
    Assume first that $t\geq R^{-1}\parens*{\frac{s+1}{2},n}\log n$. It follows from~\cref{thm:ProbabilisticLowerBound} and~\cref{eq:upper_bound_RT} that for some constant $p>0$,
    \[ p \cdot \RT{s}{t/\log n}{n} \leq g(n;s,t)\leq \RT{s}{t}{n}.\]  
    Applying~\cref{theorem:BHSphase} we obtain $g(n;s,t)=\Theta(n^2)$.
    The second point follows similarly, by observing that if $\ell\geq\frac{s+3}{2}$, then $\frac{s-1}{\ell-1}<2$ and~\cref{theorem:BHSphase} yield that  $\RT{s}{R^{-1}(\ell,n)}{n}=o(n^2)$.
\end{proof}

\section{Computational results}\label{sec:comp}

In this section, we present some computational results obtained at the start of this project, which might be of interest for future work.

Angell's Problem~\ref{DavidProblem} asks to prove that every red/blue/purple colouring of $K_8$, in which the purple edges form a perfect matching, contains a red/purple copy of $K_3$ or a blue/purple copy of $K_4$. It is not hard to prove that three purple matching edges are sufficient, while two are not. 

More generally, let us denote by $g_M(n;s,t)$ the largest $m$ such that there exists a red/blue/purple colouring of $K_n$ with a matching of exactly $m$ purple edges that has neither a red/purple copy of~$K_s$ nor a blue/purple copy of~$K_t$. If, for example, $g_M(n;s,t)=0$, then a single purple edge is sufficient to guarantee a red/purple~$K_s$ or a blue/purple~$K_t$. 

For all $s,t$ for which $R(s,t)$ is known, McKay's webpage~\cite{BrendanPage} provides the list $\mathcal{L}(n;s,t)$ of all graphs on $n=R(s,t)-1$ vertices, with no clique of size~$s$, and no independent set of size~$t$, that is, the list of red graphs in $(s,t)$-free red/blue colourings of~$K_n$. 
A crucial observation is that if \RBP{} is a red/blue/purple colouring that with no red/purple copy of~$K_s$ and no blue/purple copy of~$K_t$, then every recolouring of the purple edges by either red or blue produces a red graph in $\mathcal{L}(n;s,t)$.

\begin{observation}\label{rem_pseudo}
Fix $n>s,t\geq 3$. If $g_M(n;s,t) \geq k$, then there exist a graph~$R$ in $\mathcal{L}(n;s,t)$ and a matching~$P$ in~$K_n\setminus R$ of size~$k$ such that $R\cup P'$ is in $\mathcal{L}(n;s,t)$ for every subset $P'\subseteq P$.
\end{observation}

We remark further that it is sufficient to check that $R$ and $R\cup P$ are in $\mathcal{L}(n;s,t)$, as this trivially implies that $R\cup P'$ is in $\mathcal{L}(n;s,t)$ for every subset $P'\subseteq P$. Using SageMath~\cite{sagemath} and Nauty~\cite{mckay2014practical}, we profit from these remarks, and several other optimisations beyond brute force, to compute $g_M(n;s,t)$ for all triples $s,t,n$ where $n=R(s,t)-1$ is known. For completeness, we include pseudocode in~\cref{sec:PseudoCode}. \cref{table:gm} contains all computational results obtained for $g_M(n;s,t)$, where $(\dagger)$ denotes the assumption that the list $\mathcal{L}(42;5,5)$ is complete, which is conjectured in~\cite{r55conjecture}. If true, this would imply that $R(5,5)=43$, since none of the colourings in $\mathcal{L}(42;5,5)$ extend to a $(5,5)$-free colouring on~43 vertices, see~\cite{r55conjecture} for details.


\begin{table}[H]
\begin{center}
\begin{tabular}{ | >{\columncolor{lightgray}}l |c|c|c|c|}
\cline{2-4}
\rowcolor{lightgray}
\multicolumn{1}{c|}{\cellcolor{white}}& $t=3$ & $t=4$ & $t=5$\\ 
\cline{2-4}
$s=3$ & $ g_M(5;3,3) = 0$ &$g_M(8;3,4) = 2$ &$g_M(13;3,5) = 0$\\
\cline{2-4}
$s=4$ &  & $g_M(17;4,4) = 0$&$g_M(24;4,5) = 12$\\
\cline{2-4}
$s=5$ &  &  &  $g_M(42;5,5) = 6^{(\dagger)}$\\
\cline{2-4}
\multicolumn{1}{c}{\cellcolor{white}}\\
\cline{2-5}
\rowcolor{lightgray}
\multicolumn{1}{c|}{\cellcolor{white}}& $t=6$ & $t=7$ & $t=8$ & $t=9$\\ 
\hline
$s=3$ & $g_M(17;3,6) = 2$ & $g_M(22;3,7) = 6$ & $g_M(27;3,8) = 7$ & $g_M(35;3,9) = 0$ \\
\hline
\end{tabular}
\caption{Values of $g_M(R(s,t)-1;s,t)$.}
\label{table:gm}
\end{center}
\end{table}

We note that some of these results do not require computer aid. For example, $g_M(5,3,3)=0$ can be deduced from the fact that $C_5$ is self-complementary and the unique triangle-free graph on $5$ vertices with no independent set of size $3$. Similarly, the Paley graph on $17$ vertices is self-complementary and the unique $K_4$-free graph on $17$ vertices with no independent set of size $4$, which directly implies that $g_M(17,4,4)=0$. Some other results have simple proofs using pigeonhole arguments.

We also obtained the value of $g(n;s,t)$ for the same values of~$n$,~$s$ and~$t$, that is, the parameter studied extensively in this article, where purple edges do not necessarily form a matching. In most cases we obtain $g(n;s,t)=g_M(n;s,t)$, with the following exceptions: 
\begin{align*}
    g(27;3,8) = 9,\qquad 
    g(24;4,5) = 14,\qquad
    g(42;5,5) = 7.
\end{align*}

We mentioned in the introduction that a perfect matching of purple edges in $K_{24}$ is not sufficient to find a red/purple copy of $K_4$ or a blue/purple copy of $K_5$. Therefore, the fact that $g(24;4,5)>g_M(24;4,5)$ is not too surprising. 

The other two cases are somewhat more surprising. For $(s,t)=(3,8)$, for example, we have $R(3,8)= 28$, and we determined above that $g_M(27;3,8) = 7$. However, there is a graph $P$ with nine edges and a $(K_3,K_8)$-free red/blue/purple colouring \RBP{} of $K_{27}$. Here, we may take $P$ to be the graph consisting of vertex-disjoint copies of the 3-edge path, the 2-edge path and four pairwise disjoint edges, for example. 
A similar behaviour is observed for $(s,t)=(5,5)$ and $n=42$, assuming that $\mathcal{L}(42;5,5)$ is complete, as mentioned above. 
Here, $g_M(42;5,5) = 6$, while there is a $(K_5,K_5)$-free red/blue/purple colouring \RBP{} of $K_{42}$, where $P$ consists of a 3-edge path and four vertex-disjoint edges, that is $P$ has~$7$ edges.
These two cases exemplify that a purple matching is not necessarily optimal if we wish to maximise $|P|$ while maintaining $\omega(R\cup P)<s$ and $\alpha(R) <t$.

\section{Concluding remarks}\label{sec:Conclusion}

In this article, we initiate the systematic study of double-coloured, or purple edges, in Ramsey theory. To the best of our knowledge, this is novel. 
Naturally, there are many open problems in this context. 

\medskip

\noindent
{\bf Ramsey-Tur\'an numbers.}\\ 
The focus of this paper is the extremal parameter $g(n;s,t)$, and we have seen how this is tightly connected to the Ramsey-Tur\'an number $\RT{s}{t}{n}$.

The numbers $\RT{3}{cn}{n}$ are known asymptotically for all $c\in(0,1]\sm (1/3,4/11)$, and we expect the gap range $(1/3,4/11)$ to be covered as well soon, following the  exciting announcement of resolving \cref{conj:Andrasfai}. The outcome of the triangle-free process, see \cref{thm:OutputTriangleFreeProcess}, together with a blow-up construction similar to the one in the proof of \cref{thm:sublinear_triangle} imply that 
$$(1/4-o(1))tn\le \RT{3}{t}{n}\le tn/2$$ 
whenever $t\ge (\sqrt{2}+o(1)) \sqrt{n\log n}$. 
For $s\ge 4$, the asymptotic values of $\RT{s}{cn}{n}$ are only known for~$c$ small enough or for $c\ge 1/(s-1).$ In the light of the connection to purple edges in Ramsey theory, we would be interested in any progress of estimating
$\lim_{n\to\infty} \RT{s}{cn}{n}/n^2$ for all $0<c<1/(s-1)$. 
A lower bound of $1/4+2c/3$ on this limit is given in~\cite{fox2015critical} for $s=4$ and when $c\le 1/3$. 

While most of the classical results in Ramsey-Tur\'{a}n theory focus on $t$ being linear in $n$, not as much is known in the sublinear case. Many results have been obtained on threshold phenomena, that is, functions $f(n)$ for which $\lim_{n\to \infty}\RT{s}{t}{n}/n^2$ differs for~$t$ above and below~$f$, see for example, Kim, Kim, and Liu~\cite{kim2019two} for $s=8$ and Balogh, Chen, McCourt, and Murley~\cite{balogh2024ramsey} for~$s$ up to~$13$. For these small cases, there exists a pattern showing $\RT{s}{t}{n}=o(n^2)$ for $t\ll R^{-1}((s+1)/2,n)$, which we know to be true for general $s$ assuming~\cref{conj:OnRamsey} following~\cref{theorem:BHSphase}. We would be interested in the behaviour of $\RT{s}{t}{n}$ below these thresholds. We note that Sudakov~\cite{Sudakov2003} showed that $\RT{4}{n^{1-1/r}}{n}< n^{2-1/r(r+1)}$ for any $r\geq 2$, and also asked the general question of determining $\RT{4}{f(n)}{n}$ for any function $f(n)$.

\begin{question}\label{question:subquadratic}
    What is the order of magnitude of $\RT{s}{t}{n}$ for $s\geq 4$ when $t$ is such that $\RT{s}{t}{n}=o(n^2)$?
\end{question}

\medskip

\noindent
{\bf At the far-end of the spectrum.}\\
When $t$ is very close to $R^{-1}(s,n)$, say up to an additive constant, we expect this question to be rather difficult. At that end of the spectrum, estimating $\RT{s}{t}{n}$ becomes a question of estimating the maximum number of red edges in a nearly optimal $(s,t)$-free colouring of $K_n$. Given the complexity of estimating Ramsey numbers, this question, although a tantalising open problem, might be out of reach of current methodologies. However, understanding the difference in behaviours between $\RT{s}{t}{n}$ and $g(n;s,t)$ for such values of $t$ might not be. We have seen in~\cref{thm:TriangleCase} that $g(n;3,t)=\Omega\big(\RT{3}{t}{n}\big)$ for $t\ge (1+\eps)\sqrt{2n\log n}$. We do not expect such a result to hold when $t$ is very close to $R^{-1}(s,n)$, therefore, determining the order of magnitude of $g(n;3,t)$ when $t=R^{-1}(s,n)+1$ would be already interesting. In this regime, the results presented in~\cref{sec:comp} show that for all $t\leq 9$, a perfect matching of purple edges is sufficient to recover the Ramsey property in~$K_n$. We wonder whether this holds for general~$t$.
\begin{question}\label{conj:PM}
    Is it true that for all integers $t\geq 3$ and $n=R(3,t)-1$, every red/blue/purple colouring \RBP{} of $K_n$, where $P$ is a perfect matching, contains a red/purple triangle or a blue/purple copy of $K_t$?
\end{question}

As mentioned in the introduction and in \cref{sec:comp}, we cannot expect the answer to \cref{conj:PM} to be yes when 3 is replaced by~$s\geq 4$. We also mentioned at the end of \cref{sec:comp} that purple matchings do not necessarily have the maximal number of purple edges among all $(s,t)$-free red/blue/purple colouring of $K_n$, even when $s=3$. In that case, we would like to understand what happens when $P$ contains some other, fixed structure. Natural candidates are Hamilton cycles.
\begin{question}\label{question:Hamilton}
Is it true that for all integers $s,t\geq 3$ and $n=R(s,t)-1$, every red/blue/purple colouring \RBP{} of $K_n$, where $P$ is a Hamilton cycle, contains a red/purple copy of $K_s$ or a blue/purple copy of $K_t$?
\end{question}

\medskip

\noindent
{\bf The symmetric case when $s=t$.} \\ 
Any of these considerations on $g(n;s,t)$ can be studied in the symmetric setting, when $s=t$. Write $g(n;s)$ for $g(n;s,s)$.
Observe that $n<R(t,t)$ implies that $t>(\log n)/2$. The results of Nagy~\cite{nagy2016density} imply that $g(n; c\log n)\geq(1-\varepsilon_1)\binom{n}{2}$ for every $c>2$,  and that  $g(n; c\log n)\leq (1-\varepsilon_2)\binom{n}{2}$ for every~$c> 3/5$, where~$\varepsilon_1\approx 2/c$ and~$\varepsilon_2\approx 1/(5c)$. It would be interesting to improve these bounds, and extend them to determine $g(n;s)$ when $s=c\log n$, for any~$c>1/2$.

\paragraph{Acknowledgement}
We are grateful to David Angell for the inspiring problem at the source of this article, to Rob Morris for useful discussions on the triangle free process, to Michael Anastos, Shagnik Das, Patrick Morris and Silas Rathke for valuable feedback and remarks provided during the 2024 Workshop of the Combinatorics and Graph Theory Research Group at FU Berlin, and to Zolt\'{a}n L\'{o}r\'{a}nt Nagy for bringing~\cite{nagy2016density} to our attention.

{
\fontsize{10pt}{11pt}
\selectfont
\let\oldthebibliography=\thebibliography
\let\endoldthebibliography=\endthebibliography
\renewenvironment{thebibliography}[1]{%
\begin{oldthebibliography}{#1}%
\setlength{\parskip}{0.2ex}%
\setlength{\itemsep}{0.2ex}%
}{\end{oldthebibliography}}
\bibliographystyle{plainnat}
\bibliography{bibliography.bib}
}

\appendix

\section{Pseudocodes}\label{sec:PseudoCode}

We now present the pseudocodes used to compute results presented in~\cref{sec:comp}. A {\em Ramsey$(s,t)$-graph} is a graph with no clique of size $s$, and no independent set of size $t$. Algorithm~\ref{algorithm:1} below takes as input a Ramsey$(s,t)$-graph $G$, and outputs the largest~$k$ such that a matching of size~$k$ in $G$ can be deleted without creating an independent set of size $t$. Algorithm~\ref{algorithm:2} determines $g_M(n;s,t)$ by looping through all Ramsey$(s,t)$-graphs, using several optimisations to cut down on runtime. \medskip

\begin{algorithm}[H]
\caption{Given a Ramsey$(s,t)$-graph $G$, returns the largest~$k$ such that $\alpha(G - M) < t$ for some $k$-matching $M\subseteq E(G)$.}\label{algorithm:1}
\begin{algorithmic}[1]
\Procedure{Find\_Swapable}{\texttt{Red\_graph},$s$,$t$}
\State \texttt{List\_of\_swapable\_edges} = list()
\For{$e \in E$(\texttt{Red\_graph})}
\If {the size of the largest clique in (\texttt{Blue\_graph}$\ \cup \ \{e\}) < t$}
\State \texttt{List\_of\_swapable\_edges}.append($e$)
\EndIf
\EndFor
\State \textbf{return } \texttt{List\_of\_swapable\_edges}
\EndProcedure\\

\Procedure{Find\_k}{\texttt{Red\_graph},$s$,$t$}
\If {\Call{Find\_swapable}{\texttt{Red\_graph},s,t} $= \varnothing$ }
\State{\textbf{return} 0}
\EndIf
\State $k$ = 1
\While {$k \leq \lfloor|V(\texttt{Red\_graph})|/2\rfloor$}
\State \texttt{Is\_k\_Ramsey} = \texttt{True}
\State \texttt{matchings} = matchings of size $k$ in \Call{Find\_Swapable}{\texttt{Red\_graph},s,t}
\For {$M$ in \texttt{matchings}}
\If {the size of the largest clique in (\texttt{Blue\_graph}$\ \cup \ M) < t$}
\State \texttt{Is\_k\_Ramsey} = \texttt{False}
\State \textbf{break}
\EndIf
\EndFor
\If {\texttt{Is\_k\_Ramsey} is \texttt{True}}
\State \textbf{return} $k-1$
\EndIf
\State $k$ += 1
\EndWhile
\State \textbf{return } $\lfloor|V(\texttt{Red\_graph})|/2\rfloor$
\EndProcedure
\end{algorithmic}
\end{algorithm}

\begin{algorithm}[H]
\caption{Determine $g_M(n;s,t)$}\label{algorithm:2}
\begin{algorithmic}[1]
\Procedure{max\_k\_over\_all\_critical}{\texttt{All\_critical\_graphs},$s$,$t$,$n$}
\State \texttt{Min\_size = $n^2$}; \texttt{Max\_size} = 0
\For{$G$ in \texttt{All\_critical\_graphs}}
\If{$|E(G)| <$ \texttt{Min\_size}}
\State \texttt{Min\_size} = $|E(G)|$
\EndIf
\If{$|E(G)| >$ \texttt{Max\_size}}
\State \texttt{Max\_size} = $|E(G)|$
\EndIf
\EndFor
\State \texttt{Edge\_bound} = \texttt{Max\_size} - \texttt{Min\_size} + 1
\State \texttt{Max\_k} = 0
\For {$G$ in \texttt{All\_critical\_graphs}}
\If {$|E(G)|$ - \texttt{Max\_k} $<$ \texttt{Min\_size}}
\State \textbf{continue} (to next graph)
\EndIf
\State $k$ = \Call{Find\_k}{$G$,$s$,$t$}
\If {$k = \lfloor n/2 \rfloor$}
\State \textbf{return} $\lfloor n/2 \rfloor$
\EndIf
\If {$k > $ \texttt{Max\_k}}
\State \texttt{Max\_k} = $k$
\EndIf
\If{\texttt{Max\_k} = \texttt{Edge\_bound}} - 1
\State \textbf{break}
\EndIf
\EndFor
\State \textbf{return} \texttt{Max\_k} 
\EndProcedure
\end{algorithmic}
\end{algorithm}

\end{document}